\pdfoutput=1
\documentclass[ijoc,nonblindrev]{informs4}

\OneAndAHalfSpacedXII %
\usepackage{natbib}
 \bibpunct[, ]{(}{)}{,}{a}{}{,}%
 \def\bibfont{\small}%
 \def\bibsep{\smallskipamount}%
\TheoremsNumberedThrough    
\ECRepeatTheorems
\EquationsNumberedThrough    
\MANUSCRIPTNO{}

\usepackage{algorithm} 
\usepackage{algorithmic} 
\usepackage{tikz} 
\usetikzlibrary{calc, arrows.meta}
\definecolor{darkblue}{RGB}{31,78,121}
\definecolor{medblue}{RGB}{68,114,196}
\definecolor{darkred}{RGB}{192,0,0}
\definecolor{darkgreen}{RGB}{0,128,0}
\definecolor{amber}{RGB}{255,192,0}
\definecolor{conefill1}{RGB}{255,230,200}
\definecolor{conefill2}{RGB}{220,235,255}
\definecolor{projcolor}{RGB}{148,103,189}
\usepackage{enumerate}
\usepackage{rotating,lineno}
\usepackage{flafter}
\usepackage[T1]{fontenc}
\usepackage[utf8]{inputenc}
\usepackage{babel}
\usepackage[font=small,labelfont=bf]{caption}
\usepackage{graphicx}
\usepackage{capt-of}%
\usepackage[mathscr]{euscript}
\usepackage{pgfplots}
\usepackage{multirow}
\usepackage{array, makecell}
\usepackage{booktabs}
\usepackage{hyperref}
\hypersetup{
    colorlinks=true,
    linkcolor=blue,
    filecolor=magenta,      
    urlcolor=blue,
    citecolor={blue},
}
\pgfplotsset{compat=1.12}

\DeclareSymbolFont{rsfs}{U}{rsfs}{m}{n}
\DeclareSymbolFontAlphabet{\mathscrsfs}{rsfs}

\definecolor{lightgray}{RGB}{229, 231, 233}

\usepackage{amsmath}

\usepackage{nicefrac}

\usepackage[shortlabels]{enumitem}
\usepackage{subcaption}

\newcommand{\PIOname}{Pooled Inverse Optimization}

\newcommand{\PIO}{\text{PIO}}

\newcommand{\ICO}{\text{ICO}}

\newcommand{\ILO}{\text{ILO}}
\newcommand{\FO}{\text{FO}}

\tikzstyle{very loosely dotted}=          [dash pattern=on \pgflinewidth off 6pt]
\tikzstyle{very loosely dashdotted}=      [dash pattern=on 3pt off 8pt on \the\pgflinewidth off 8pt]
\tikzstyle{costumdashed}=                  [dash pattern=on 10pt off 3pt]

\begin{document}
\newcommand{\Kc}{\mathcal{K}}
\newcommand{\Zstar}{\mathcal{Z}^{*}}
\newcommand{\Nex}{\mathcal{N}(\Omega)}
\newcommand{\interior}{\operatorname{int}}
\newcommand{\cone}{\operatorname{cone}}
\newcommand{\conv}{\operatorname{conv}}
\newcommand{\rec}{\operatorname{rec}}
\newcommand{\proj}{\operatorname{proj}}
\newcommand{\bdy}{\operatorname{bd}}

\RUNTITLE{\PIOname}

\TITLE{\PIOname{}: Learning the Unique Optimal Decision from Repeated Observations}
\ARTICLEAUTHORS{%
\AUTHOR{Farzin Ahmadi}
\AFF{Department of Health Sciences, Towson University, Towson, MD 21252; and Data Science and AI Institute, Johns Hopkins University, Baltimore, MD 21218, \EMAIL{fahmadi@towson.edu}}
\AUTHOR{Fardin Ganjkhanloo}
\AFF{Center for Health Systems and Policy Modeling, Johns Hopkins University, Washington, DC 20001, \EMAIL{fganjkh1@jhu.edu}}
\AUTHOR{Kimia Ghobadi}
\AFF{Department of Civil and Systems Engineering, Johns Hopkins University, Baltimore, MD 21218, \EMAIL{kimia@jhu.edu}}
}

\RUNAUTHOR{Ahmadi, Ganjkhanloo, and Ghobadi}
\ABSTRACT{%
Inverse optimization infers a decision-maker's objective from observed decisions. We study repeated noisy observations of one decision. Classical formulations impute a separate optimum for each observation, so their size grows with the sample, and may return one of many equally compatible objectives. With a bounded polyhedral feasible set and a cost vector drawn from a distribution with a density, the optimum is almost surely a unique vertex. Its compatible costs form a full-dimensional cone that repeated observations of the same decision cannot shrink. Pooled inverse optimization fits one optimum to all observations and reports its compatible cost cone. Under squared loss, the optimization model has fixed size after aggregation, the decision estimate has a finite-sample error bound, and guarded rounding recovers the true vertex and its full cone with probability tending to one.
A single convex program sufficed in every linear scaling run. With 50 variables and 10,000 observations, median runtimes were 8 milliseconds under squared loss and 8 seconds under absolute loss. The standard mixed-integer formulation reached its time limit at 100 observations. At 10,000 observations, the remaining two exact formulations under absolute loss and two convex baselines reached their 10- to 60-minute limits. Using national survey data on women with hypertension and diabetes or prediabetes, the method finds in milliseconds the diet closest to each woman's mean reported intake subject to DASH-based nutrient bounds in a food-group model and identifies which bounds bind.
}
\AREAOFREVIEW{Design \& Analysis of Algorithms--Continuous}
\maketitle

\section{Introduction}
\label{sec:introduction}

When a decision-maker solves an optimization problem whose constraints are known but whose objective is not, \emph{inverse optimization} seeks to recover the objective from observed solutions. Since \citet{Ahuja01}, the area has matured into a standard tool with applications in healthcare \citep{chan2014generalized, ajayi2022objective}, transportation \citep{chow2012inverse}, energy \citep{saez2017short}, cooperative games \citep{liu2024stabilizing}, and finance \citep{bertsimas2012inverse}; \citet{chan2023inverse} survey the field.

This paper concerns a data regime that is common in practice and specific in structure: \emph{one decision-maker facing a fixed decision problem is observed repeatedly}. A patient with hypertension records daily food intake for months; an operator dispatches the same fleet over the same network every day. The observations are noisy views of one latent decision. Two difficulties arise when classical data-driven inverse optimization \citep{keshavarz2011imputing, bertsimas2015data, aswani2018inverse} is applied here. First, the decision-space formulations \citep{aswani2018inverse, babier2021ensemble} impute one optimal solution per observation and enforce the optimality conditions at each, so their size grows linearly with the number of observations $K$; with complementarity handled by binary variables, this means $Km$ binaries for a forward problem with $m$ constraints. Second, the objective these formulations return is rarely identified: at a vertex of a polyhedron every cost vector in a full-dimensional cone rationalizes the same decision, and repeated observations of that decision cannot shrink the cone \citep{Ahuja01, bodur2022inverse}. Recent work resolves the ambiguity by selecting a canonical element, such as the circumcenter \citep{besbes2023contextual} or the incenter \citep{zattoni2025learning}, or an element closest to a reference \citep{gupta2022decomposition}.

When observations are noisy measurements of a common latent decision, a single imputed decision suffices for the entire sample. For a linear forward problem, the geometry of the feasible set determines which costs can rationalize that decision. 
For a bounded polyhedral feasible set and a cost vector drawn from a distribution with a density, the forward linear program has a unique optimal solution, and that solution is a vertex, with probability one (Theorem~\ref{thm:generic-uniqueness}): the cost vector is almost never orthogonal to an edge of the polyhedron. The fact is known in various forms \citep{aswani2018inverse, elmachtoub2022smart}; we state it in the form we need. At a vertex, the set of cost vectors for which the vertex is optimal is the cone generated by the active constraint normals, which is full-dimensional (Proposition~\ref{prop:uniqueness-interior}). So at a unique optimal decision the objective is not identified, and repeated observations of that decision do not change this. In the repeated-decision regime, the data identify the decision and the cone of objectives compatible with it, and these are what we estimate.

\emph{\PIOname} ($\PIO$) estimates the decision and reports its cone. It is the classical formulation with the $K$ imputed optima tied to a single point $z$, and we do not claim that the tie is new: \citet{gupta2022decomposition} share one imputed vertex across the noisy samples of an experiment, \citet{ahmadi2022you} impute one representative optimum for each cluster of a heterogeneous population of decision-makers, and aggregating similar variables is a familiar way to make large planning models tractable \citep{ungun2019real}. With the squared loss, $\PIO$ is the single-observation decision-space inverse problem applied to the centroid of the data: that of \citet{chan2019inverse} when the centroid is feasible, and that of \citet{babier2021ensemble} with one observation in general. Our contribution is the analysis of the tied model. The number of variables and constraints of $\PIO$ does not depend on $K$, apart from the breakpoints of the absolute loss (Theorem~\ref{thm:complexity}), and the model reports the whole cone of objectives compatible with its decision. Its decision $\hat z$ satisfies $\mathbb{E}\|\hat z - z_0\|_2 \le 2\sigma/\sqrt{K}$ for every $K$ under the squared loss, where $z_0$ is the true decision and $\sigma^2$ the expected squared norm of the noise, and is consistent under the absolute loss when the feasible set is bounded (Theorem~\ref{thm:consistency}), and, for linear forward problems, the reported cone eventually lies inside the true one (Theorem~\ref{thm:cone-recovery}). A second stage that rounds $\hat z$ to a vertex of its face, guarded by a loss test, recovers a vertex $z_0$ exactly with probability tending to one under the squared loss, and with it the whole cone (Theorem~\ref{thm:twostage}). For linear forward problems, one convex program over the feasible set solves $\PIO$ whenever no unconstrained minimizer of the loss is feasible, and $m$ convex programs, one per constraint, always do (Proposition~\ref{prop:linear-structure}).

\subsection{Contributions}

\textbf{1. Unique decisions, non-unique objectives.} For a pointed polyhedron and a cost vector with a density, the forward problem is, with probability one, unbounded or has a unique optimal vertex, at which the compatible objectives form a full-dimensional cone (Theorem~\ref{thm:generic-uniqueness}, Proposition~\ref{prop:uniqueness-interior}, Corollary~\ref{cor:non-uniqueness}). We state exactly when the sign-free and simplex normalizations leave the objective unique, and give the corresponding rank condition for parametric convex objectives.

\textbf{2. An estimator with a fixed number of variables and constraints.} Section~\ref{sec:PIO} formulates $\PIO$ for parametric convex forward problems, proves its size, solution, and statistical properties, and adds the rounding stage that recovers the optimal vertex; the proofs are in the electronic companion.

\textbf{3. Computational evidence.} Section~\ref{Section:NumericalEx} compares $\PIO$ with three exact formulations and two convex approximations of the classical model for $K$ up to $10^4$, on linear and convex quadratic forward problems and across noise levels and binding structures; Section~\ref{Section:Application} applies $\PIO$ to NHANES diet records.

\subsection{Relevant Literature}

The classical strand of inverse optimization \citep{Ahuja01, heuberger2004inverse, chan2014generalized} recovers parameters under the assumption that observations are exactly optimal, and the polyhedral structure of the set of compatible objectives is well understood: it is a cone generated by the active constraint normals for linear programs \citep{chan2019inverse} and, for mixed-integer forward problems, a cone described by the inequalities of a possibly small set of feasible points \citep{bodur2022inverse}. The data-driven strand accommodates noisy observations through loss-based formulations \citep{keshavarz2011imputing, bertsimas2015data, aswani2018inverse, esfahani2018data}; \citet{chan2019inverse} give closed-form solutions and goodness-of-fit measures for a single observation of a linear program, and \citet{tavaslioglu2018structure} characterize the set of cost vectors that make a feasible point optimal. \citet{babier2021ensemble} extend the geometric view to multiple observations: they solve the multi-observation decision-space problem, our $\ILO$ (Section~\ref{subsec:prelim}), with one dual vector shared by all observations, through $m$ convex problems, one per constraint, and show that for duality-gap losses and feasible observations the multi-observation model is equivalent to the single-observation model at the centroid. Proposition~\ref{prop:linear-structure} is the analogue of these results for the tied model. \citet{shahmoradi2021quantile} stabilize multi-observation inverse linear optimization against outliers through quantiles. \citet{gupta2022decomposition} is the closest predecessor of the present paper. For a dataset with several experiments, each observed with several noisy samples, they impute one vertex per experiment shared by its samples, recover the full admissible cone, and design a decomposition whose feasibility subproblem has a number of binaries independent of the number of data points. They assume that optimal solutions lie at vertices, which Theorem~\ref{thm:generic-uniqueness} justifies for almost every cost vector; in the single-instance regime, where all observations share one decision, two-stage $\PIO$ reaches the vertex through convex programs and a rounding step, with finite-sample and consistency guarantees under the squared loss. Our emphasis on the decision parallels the distinction of \citet{aswani2018inverse} between risk consistency and estimation consistency, only the latter of which needs an identifiability condition. Contextual and online settings, in which the forward instance varies across observations, are treated by \citet{besbes2023contextual}, \citet{zattoni2025learning}, and \citet{lin2024conformal}; that setting is not treated here.

Decision-focused learning \citep{elmachtoub2022smart, loke2026decision} observes costs and learns a predictor that minimizes decision regret; inverse optimization observes decisions instead, but both judge success in decision space, as our study does. A line of work in this journal builds computational tools for clinical decision support in which run time and model size are binding \citep{ahmed2010automated, ungun2019real, ajdari2020theoretical}; our diet application is in that tradition.

\section{Setting: Unique Decisions and Non-Unique Objectives}
\label{sec:setting}

\subsection{Forward Problem and Classical Inverse Formulations}
\label{subsec:prelim}

We consider a structured parametric objective
\begin{equation}\label{eq:parametric-obj}
f(x,\theta) = \sum_{j=1}^p \theta_j \phi_j(x),
\end{equation}
where $\theta \in \mathbb{R}^p$ is unknown and $\phi_j:\mathbb{R}^n \to \mathbb{R}$ are known basis functions. The feasible set is
\begin{equation}\label{eq:Omega}
\Omega = \{x \in \mathbb{R}^n : g_i(x) \leq 0,\ i = 1,\ldots,m\},
\end{equation}
with known $g_i$, and the \emph{forward problem} for $\theta$ is
\begin{equation}\label{eq:FO}
\FO(\theta,\Omega):\quad \min_{x \in \Omega}\ f(x,\theta).
\end{equation}
Throughout, $\Omega$ and the basis functions are fixed: every observation is a noisy solution of the same forward instance, and only $\theta$ is unknown. The \emph{linear case} is $\phi_j(x) = x_j$ ($p = n$) and $g_i(x) = b_i - a_i^\top x$, so that $f(x,\theta) = \theta^\top x$ and $\Omega = \{x : Ax \ge b\}$ with $A \in \mathbb{R}^{m \times n}$ having rows $a_i^\top$; the \emph{convex quadratic case} is $\phi_j(x) = \tfrac12 (x-c_j)^\top Q_j (x-c_j)$ with $Q_j \succeq 0$. We write $\Phi(z) := [\nabla\phi_1(z)\ \cdots\ \nabla\phi_p(z)] \in \mathbb{R}^{n \times p}$ for the matrix of basis-function gradients, so that $\nabla_x f(z,\theta) = \Phi(z)\theta$; in the linear case $\Phi(z) = I_n$. For $z \in \Omega$, $I(z) := \{i : g_i(z) = 0\}$ is the active set.

\begin{assumption}%
\label{assump:standing}
(i) The functions $\phi_j$ and $g_i$ are convex and continuously differentiable on $\mathbb{R}^n$. (ii) $\Omega$ satisfies Slater's condition: there is $\hat x$ with $g_i(\hat x) < 0$ for all $i$. (iii) The admissible parameter set $\Theta \subset \mathbb{R}^p$ is compact, and $f(\cdot,\theta)$ is convex for every $\theta \in \Theta$; this holds for $\Theta \subseteq \mathbb{R}^p_+$, and for every $\Theta$ when all $\phi_j$ are affine.
\end{assumption}

The normalization $\Theta$ removes the scale invariance of the optimality conditions and excludes $\theta = 0$. We use two: the sign-free normalization $\Theta = \{\theta : \|\theta\|_2 = 1\}$ in the linear case, and the simplex normalization $\Theta = \{\theta \ge 0 : \mathbf{1}^\top\theta = 1\}$ when the basis functions are nonlinear. Under Assumption~\ref{assump:standing}, for $\theta \in \Theta$ a point $z \in \Omega$ is optimal for $\FO(\theta,\Omega)$ if and only if the Karush--Kuhn--Tucker (KKT) conditions hold: there is $\lambda \in \mathbb{R}^m_+$ with $\Phi(z)\theta + \sum_i \lambda_i \nabla g_i(z) = 0$ and $\lambda_i g_i(z) = 0$ for all $i$ \citep[Theorems 28.2--28.3]{rockafellar1970convex}.

Given observations $\mathcal{X}=\{x^k\}_{k=1}^K$, classical inverse convex optimization imputes an optimal solution $z^k$ near each $x^k$ and one parameter that makes all of them optimal:
\begin{align}
\ICO(\mathcal{X},\Omega):\quad
\min_{\theta,\{\lambda^k\},\{z^k\}} \ & \sum_{k=1}^K \ell(x^k - z^k) \label{eq:ICO}\\
\text{s.t.}\quad
& g_i(z^k) \le 0,\quad \lambda^k_i g_i(z^k) = 0,\quad \lambda^k \ge 0,\qquad \forall i, k, \nonumber\\
& \Phi(z^k)\theta + \sum_{i=1}^m \lambda_i^k \nabla g_i(z^k) = 0,\quad \forall k, \qquad \theta \in \Theta, \nonumber
\end{align}
where $\ell$ is a convex loss (Remark~\ref{rem:loss}). In the linear case \eqref{eq:ICO} becomes
\begin{align}
\ILO(\mathcal{X},\Omega): \quad
\min_{\theta,\{\lambda^k\},\{z^k\}} \;& \sum_{k=1}^K \ell(x^k-z^k) \label{eq:ILO}\\
\text{s.t.}\quad & Az^k \ge b,\ \ \theta = A^\top\lambda^k,\ \ \lambda^k \ge 0,\ \ \lambda^k_i (a_i^\top z^k - b_i) = 0\ \ \forall i,k,\ \ \theta \in \Theta. \nonumber
\end{align}
Two facts about these formulations are standard. Computationally, they are nonconvex because of the complementarity conditions, and they grow linearly in $K$ (Theorem~\ref{thm:complexity}). Statistically, the estimator of \citet{aswani2018inverse} is consistent for the parameter under an identifiability condition (the residual-based estimators of \citealt{keshavarz2011imputing} and \citealt{bertsimas2015data} are not; \citealt[Proposition~1]{aswani2018inverse}), but consistency of a parameter estimator presupposes that the parameter is determined by the data, and for a fixed forward instance it typically is not. We now make this precise for linear forward problems: when the optimal decision is unique it is a vertex, and at a vertex the parameter is not identified.

\subsection{The Optimal Decision Is Generically a Unique Vertex}
\label{subsec:uniqueness}

We work in the linear case, $\Omega = \{x : Ax \ge b\}$ nonempty. A polyhedron is \emph{pointed} if it contains no line; for $\Omega = \{Ax \ge b\}$ this holds if and only if $\operatorname{rank}(A) = n$ \citep[Section 8.2]{schrijver1986theory}, and in particular whenever $\Omega$ is bounded or has box constraints. A pointed polyhedron has finitely many vertices, and its recession cone $\rec(\Omega) = \{d : Ad \ge 0\}$ is pointed with finitely many extreme rays \citep[Corollary 7.1b and Sections 8.5 and 8.8]{schrijver1986theory}. For $z \in \Omega$ we write
\begin{equation}\label{eq:active-cone}
\Kc(z) := \cone\{a_i : i \in I(z)\}
\end{equation}
for the cone of active constraint normals (with $\cone\,\emptyset = \{0\}$); $\Kc(z) = -N_\Omega(z)$, the negative of the normal cone of $\Omega$ at $z$.

\begin{theorem}[Generic uniqueness of the optimal decision]\label{thm:generic-uniqueness}
Let $\Omega$ be a nonempty pointed polyhedron with vertex set $V$ and a finite set $R$ containing one direction of each extreme ray of $\rec(\Omega)$. Let $D := \{v - w : v, w \in V,\ v \ne w\} \cup R$ and let
\[
\Nex := \bigcup_{d \in D} \{\theta \in \mathbb{R}^n : \theta^\top d = 0\}
\]
be the finite union of the hyperplanes through the origin orthogonal to these directions ($\Nex = \emptyset$ if $D = \emptyset$). For every $\theta \notin \Nex$ such that $\inf\{\theta^\top x : x \in \Omega\}$ is finite, the forward problem has a unique optimal solution, and it is a vertex of $\Omega$. Consequently, if $\theta$ is a random vector with $\mathbb{P}(\theta^\top d = 0) = 0$ for every $d \ne 0$ (for instance, if $\theta$ has a density, or if $\theta = g/\|g\|$ for a vector $g$ with a density), then $\theta \notin \Nex$ with probability one; if $\Omega$ is in addition bounded, the forward problem has a unique optimal solution, a vertex, with probability one.
\end{theorem}

\begin{proposition}%
\label{prop:uniqueness-interior}
Let $\Omega = \{x : Ax \ge b\}$ be nonempty, $z \in \Omega$, and $\theta \in \mathbb{R}^n$.
\begin{enumerate}[(i)]
\item $z$ is optimal for $\min\{\theta^\top x : x \in \Omega\}$ if and only if $\theta \in \Kc(z)$.
\item $z$ is the unique optimal solution if and only if $\theta \in \interior \Kc(z)$ \citep{mangasarian1979uniqueness, tavaslioglu2018structure}. This is a statement about the decision: by (i), every point of $\Kc(z)$, interior or not, makes $z$ optimal.
\end{enumerate}
In particular, a unique optimal decision requires $\Kc(z)$ to be full-dimensional, i.e., $\operatorname{rank}\{a_i : i \in I(z)\} = n$, which is exactly the statement that $z$ is a vertex of $\Omega$.
\end{proposition}

\begin{corollary}[The objective is not identified at a vertex]\label{cor:non-uniqueness}
Let $z$ be a vertex of $\Omega = \{x : Ax \ge b\}$. Then the set $\{\theta \in \mathbb{R}^n : z \text{ is optimal for } \theta\} = \Kc(z)$ is a polyhedral cone of dimension $n$; it depends only on $z$, so it is the same whether $z$ is observed once or any number of times, and the admissible parameters for which $z$ is a common imputed optimum of $\ICO$ form the set $\Kc(z) \cap \Theta$. Under the sign-free normalization and $n \ge 2$, $\Kc(z) \cap \{\|\theta\|_2 = 1\}$ contains a relatively open subset of the unit sphere and hence infinitely many admissible parameters. Under the simplex normalization, $\Kc(z) \cap \{\theta \ge 0, \mathbf{1}^\top\theta = 1\}$ contains more than one point if and only if $\dim(\Kc(z) \cap \mathbb{R}^n_+) \ge 2$; this can fail even though $\dim \Kc(z) = n$.
\end{corollary}

For $\Omega = \{x \in \mathbb{R}^2 : x_1 \ge 0,\ x_1 - x_2 \ge 0\}$ and $z = 0$, $\Kc(z) = \cone\{(1,0),(1,-1)\}$ is two-dimensional but $\Kc(z) \cap \mathbb{R}^2_+ = \{(t,0) : t \ge 0\}$, so the simplex normalization leaves a unique parameter. Approximate formulations that relax the optimality conditions by a tolerance or minimize their residuals \citep{keshavarz2011imputing, bertsimas2015data} do not restore uniqueness: every point of $\Kc(z) \cap \Theta$ satisfies the exact conditions at $z$, so relaxing them can only enlarge the set of admissible parameters.

For parametric convex objectives, the decision is unique whenever $f(\cdot,\theta)$ is strictly convex on $\Omega$, as for the quadratic case with $\sum_j \theta_j Q_j \succ 0$. Whether the parameter is then unique depends on a rank condition and on the normalization (Proposition~\ref{prop:param-uniqueness}); as in Corollary~\ref{cor:non-uniqueness}, this depends on the geometry at the optimum and does not change with the sample size.

\section{\PIOname{}}
\label{sec:PIO}

In the repeated-decision regime the observations are noisy copies of one decision, so \PIOname{} ($\PIO$) imputes a \emph{single} optimal solution $z$ that best fits the data and reports the parameters for which it is optimal. Formally, $\PIO$ is $\ICO$ with $z^1 = \cdots = z^K$.

\subsection{Formulation and Size}

\begin{subequations}\label{eq:PIO}
\begin{align}
\PIO(\mathcal{X}, \Omega): \quad \min_{z, \theta, \lambda} \quad & \sum_{k=1}^K \ell(x^k - z) \label{eq:PIO-obj}\\
\text{s.t.} \quad & g_i(z) \leq 0, \quad i = 1, \ldots, m, \label{eq:PIO-primal}\\
& \Phi(z)\theta + \sum_{i=1}^m \lambda_i \nabla g_i(z) = 0, \label{eq:PIO-stat}\\
& \lambda_i g_i(z) = 0, \quad i = 1, \ldots, m, \label{eq:PIO-comp}\\
& \lambda \geq 0, \quad \theta \in \Theta. \label{eq:PIO-dual}
\end{align}
\end{subequations}
Constraints \eqref{eq:PIO-primal}--\eqref{eq:PIO-dual} are the KKT conditions of $\FO(\theta,\Omega)$ at $z$, so under Assumption~\ref{assump:standing} the feasible values of $z$ are exactly the points of
\begin{equation}\label{eq:Zstar}
\Zstar := \{z \in \Omega : z \text{ is optimal for } \FO(\theta,\Omega) \text{ for some } \theta \in \Theta\},
\end{equation}
the set of \emph{rationalizable} decisions, and $\PIO$ returns a minimizer $\hat z$ of the loss over $\Zstar$.

\begin{remark}%
\label{rem:loss}
We state the results for the squared loss $\ell(u) = \|u\|_2^2$ and for the absolute loss $\ell(u) = \|u\|_1$, which keeps every subproblem linear or mixed-integer linear in the linear case and is used in Section~\ref{Section:NumericalEx} unless the squared loss is named; Section~\ref{Section:Application} uses the squared loss. The structural statements (Theorem~\ref{thm:complexity} and Proposition~\ref{prop:linear-structure}(i), (ii), and (iv)) hold for any convex loss, with minima read as infima.
\end{remark}

\begin{theorem}%
\label{thm:complexity}
Apart from the representation of the loss and the constraints $\lambda \ge 0$ and $\theta \in \Theta$, $\PIO(\mathcal{X},\Omega)$ has $n + p + m$ decision variables and $n + 2m$ constraints, independent of $K$; $\ICO(\mathcal{X},\Omega)$ has $Kn + p + Km$ variables and $Kn + 2Km$ constraints. In the linear case, handling complementarity by one binary variable per condition (with valid big-$M$ bounds, or through indicator constraints) requires $m$ binaries in $\PIO$ and $Km$ in $\ILO$ as written; $\ILO$ can be reformulated exactly with one multiplier vector shared by all observations and $m$ binaries, but it keeps $Kn$ imputed coordinates and $Km$ complementarity conditions.
\end{theorem}

The loss adds no per-observation variables either, because the data can be aggregated first.

\begin{proposition}%
\label{prop:aggregation}
For the squared loss, $\sum_{k=1}^K \|x^k - z\|_2^2 = K\|z - \bar{x}\|_2^2 + \sum_{k=1}^K \|x^k - \bar{x}\|_2^2$ with $\bar x := \frac1K\sum_k x^k$, so $\PIO$ depends on the data only through the centroid $\bar x$. For the absolute loss, $\sum_{k=1}^K \|x^k - z\|_1 = \sum_{i=1}^n \sum_{k=1}^K |x^k_i - z_i|$ is separable across coordinates, each term is a convex piecewise-linear function of $z_i$ with breakpoints at the sorted sample values and unconstrained minimizers at the sample medians, and the objective can be passed to a solver as $n$ piecewise-linear terms without per-observation variables.
\end{proposition}

The model is built once from $\bar x$ (or the sorted coordinates) and updated as observations arrive; $K$ enters only via at most $nK$ breakpoints of the absolute loss.

\subsection{What $\PIO$ Returns: The Decision and Its Cone}

\begin{theorem}%
\label{thm:param-set}
Under Assumption~\ref{assump:standing}, for every $z \in \Omega$ the set of admissible parameters for which $z$ is optimal is
\begin{equation}\label{eq:theta-set}
\Theta^*(z) = \bigl\{ \theta \in \Theta : \Phi(z)\theta \in -N_\Omega(z) \bigr\},\qquad N_\Omega(z) = \cone\{\nabla g_i(z) : i \in I(z)\},
\end{equation}
a compact set, convex whenever $\Theta$ is, and a polyhedron whenever $\Theta$ is polyhedral. In the linear case $\Theta^*(z) = \Kc(z) \cap \Theta$.
\end{theorem}

$\PIO$ returns $\hat z$ and, through \eqref{eq:theta-set}, the entire set $\Theta^*(\hat z)$; the particular $\theta$ in the solver's output is one element of it, and any other element is equally valid.

In the linear case the geometry of $\Zstar$ gives an exact algorithm. Let $F_i := \{z \in \Omega : a_i^\top z = b_i\}$ be the face of $\Omega$ cut out by constraint $i$ (a facet when the constraint is irredundant, a lower-dimensional face or the empty set otherwise).

\begin{proposition}[Structure of linear $\PIO$]\label{prop:linear-structure}
Let $\Omega = \{x : Ax \ge b\}$ be nonempty with all rows $a_i \ne 0$, $\phi_j(x) = x_j$, and $\Theta = \{\theta : \|\theta\|_2 = 1\}$.
\begin{enumerate}[(i)]
\item $\Zstar = \bigcup_{i=1}^m F_i$, and for any convex loss $\ell$,
\[
\inf\Bigl\{\textstyle\sum_{k} \ell(x^k - z) : (z,\theta,\lambda) \text{ feasible for } \eqref{eq:PIO}\Bigr\}
\;=\; \min_{i = 1,\ldots,m}\ \inf_{z \in F_i} \textstyle\sum_{k} \ell(x^k - z),
\]
where empty faces contribute $+\infty$; for the squared and absolute losses every minimum is attained, and a minimizer $z^*$ on a best face $F_{i^*}$ together with $\theta^* = a_{i^*}/\|a_{i^*}\|_2$ and $\lambda^* = e_{i^*}/\|a_{i^*}\|_2$ solves \eqref{eq:PIO}. Each inner problem is a convex quadratic program of size $O(n+m)$ for the squared loss and, for the absolute loss, a linear program with $n$ auxiliary variables and $O(nK + m)$ constraints (or $n$ native piecewise-linear terms), with no per-observation variables.
\item The same enumeration solves $\ILO$ \eqref{eq:ILO}, with the inner problem $\sum_k \min_{z^k \in F_i} \ell(x^k - z^k)$ separating into $K$ independent projections per face.
\item For the squared loss, if $\bar x \notin \Omega$ then $\hat z = \proj_\Omega(\bar x)$, the Euclidean projection of the centroid onto $\Omega$, and it is unique.
\item (Relax and check.) For any convex loss, let $z_R$ minimize $\sum_k \ell(x^k - z)$ over $\Omega$. If some constraint is active at $z_R$, then $z_R$ solves $\PIO$. This is the case whenever no unconstrained minimizer of the loss (for the squared loss $\bar x$; for the absolute loss, any vector of coordinatewise medians) lies in $\Omega$.
\end{enumerate}
\end{proposition}

Part~(ii) is the face enumeration that is known in the literature \citep{babier2021ensemble} for $\ILO$, and part~(i) is its analogue for the tied model, with one projection per face instead of $K$. Part~(iii) says what $\PIO$ does when the centroid is infeasible: it moves the observed behavior to the nearest feasible point, which is automatically optimal for every objective in the cone of the constraints it activates. Part~(iv) turns this into an algorithm: solve one convex program over $\Omega$, check whether a constraint is active at its solution, and enumerate faces only otherwise. When the basis functions are nonlinear, $\Phi(z)$ varies with $z$, the rationalizable set is not a union of faces, and the mixed-integer KKT formulation is the general method. The same holds for linear forward problems under a sign restriction such as $\theta \ge 0$: a boundary point $z$ is then rationalizable only if $\Kc(z)$ contains a nonzero $\theta \ge 0$, so $\Zstar$ can be a strict subset of the boundary of $\Omega$.

\subsection{Statistical Guarantees}
\label{subsec:stat}

\begin{assumption}%
\label{assump:DGP}
Observations are $x^k = z_0 + \xi^k$, $k = 1,\ldots,K$, where $z_0$ is optimal for $\FO(\theta_0,\Omega)$ for some $\theta_0 \in \Theta$ (so $z_0 \in \Zstar$) and $\xi^1, \xi^2, \ldots$ are i.i.d.\ random vectors. For the squared loss we assume $\mathbb{E}[\xi^k] = 0$ and $\sigma^2 := \mathbb{E}\|\xi^k\|_2^2 < \infty$. For the absolute loss we assume $\mathbb{E}\|\xi^k\|_1 < \infty$ and that each coordinate $\xi^k_i$ has unique median $0$.
\end{assumption}

Beyond independent sampling, a finite moment, and centering, the assumption places no restriction on the shape of the noise distribution. In the linear case with $z_0$ the unique optimum, zero-mean noise that is not identically zero puts observations outside $\Omega$ with positive probability: if every $x^k$ were feasible almost surely, optimality of $z_0$ would give $\theta_0^\top(x^k - z_0) \ge 0$, zero mean would force equality, and every observation would lie in the optimal face $\{z_0\}$. When recorded decisions are always feasible, the centroid then converges to a point of $\Omega$ other than $z_0$, and $\hat z$ need not converge to $z_0$.

\begin{theorem}[Consistency and a finite-sample bound]\label{thm:consistency}
Let Assumptions~\ref{assump:standing} and \ref{assump:DGP} hold, and let $\hat z_K$ be any measurable selection of minimizers of \eqref{eq:PIO}.
\begin{enumerate}[(i)]
\item (Squared loss.) For every realization of the data, $\|\hat z_K - z_0\|_2 \le 2\|\bar x_K - z_0\|_2$. Consequently $\mathbb{E}\|\hat z_K - z_0\|_2 \le 2\sigma/\sqrt{K}$ for every $K$, and $\hat z_K \to z_0$ almost surely.
\item (Absolute loss.) If $\Omega$ is bounded, $\hat z_K \to z_0$ almost surely.
\end{enumerate}
\end{theorem}

By (i), $K \ge 4\sigma^2/\varepsilon^2$ observations bring the expected decision error below $\varepsilon$ whatever the geometry of $\Omega$; the proof uses only that $z_0$ is feasible for $\PIO$.

\begin{theorem}[Validity of the reported cone]\label{thm:cone-recovery}
Let Assumptions~\ref{assump:standing} and \ref{assump:DGP} hold, and let $\hat z_K \to z_0$ almost surely (as in Theorem~\ref{thm:consistency}).
\begin{enumerate}[(i)]
\item (Linear case.) There is a neighborhood $U$ of $z_0$ such that $I(z) \subseteq I(z_0)$, hence $\Theta^*(z) \subseteq \Theta^*(z_0)$, for every $z \in \Omega \cap U$. Consequently, almost surely, $\Theta^*(\hat z_K) \subseteq \Theta^*(z_0)$ for all sufficiently large $K$: every objective reported by $\PIO$ eventually rationalizes the true decision.
\item (General case.) The map $z \mapsto \Theta^*(z)$ is outer semicontinuous on $\Omega$: if $z_\nu \to z$, $\theta_\nu \in \Theta^*(z_\nu)$, and $\theta_\nu \to \theta$, then $\theta \in \Theta^*(z)$. Consequently $\limsup_{K\to\infty} \Theta^*(\hat z_K) \subseteq \Theta^*(z_0)$ almost surely.
\end{enumerate}
If $\Theta^*(z_0) = \{\theta_0\}$ (as under Proposition~\ref{prop:param-uniqueness} with a normalization $\Theta$ that meets each open ray in at most one point), then every selection $\hat\theta_K \in \Theta^*(\hat z_K)$ converges to $\theta_0$ almost surely.
\end{theorem}

When $z_0$ is a vertex, the reverse containment in (i) generally fails. For large $K$, $\hat z_K$ lies on a face through $z_0$ and reports the cone of that face, which equals $\Kc(z_0)$ only if $\hat z_K = z_0$; under the squared loss with $n \ge 2$ this is the event $\bar x_K \in z_0 - \Kc(z_0)$. If $\xi^1$ has a nonsingular covariance matrix $\Sigma$, its probability tends to $\mathbb{P}(G \in -\Kc(z_0))$ with $G \sim \mathcal{N}(0,\Sigma)$, and equals it for every $K$ under Gaussian noise. Unless $\Omega = \{z_0\}$ the cone lies in a half-space, so this limit is at most one half; on the vertex-scenario instances of Section~\ref{subsec:accuracy} it averages $0.1$--$2.9\%$. 
For all sufficiently large $K$, almost surely, the cone reported by the unrounded estimator is an inner approximation of $\Kc(z_0)$. When $\hat z_K \ne z_0$ and $z_0$ is a vertex, this approximation is a strict subcone and excludes every cost in the interior of $\Kc(z_0)$. The full cone consists of costs for which $z_0$ is optimal. Its interior consists of costs for which $z_0$ is uniquely optimal.

\subsection{Two-Stage $\PIO$: Recovering the Vertex}
\label{subsec:twostage}

Since $z_0$ is generically a vertex (Theorem~\ref{thm:generic-uniqueness}), a second stage that moves $\hat z$ to a vertex of its face can recover $z_0$ exactly. In the linear case with all rows $a_i \ne 0$, let $s_i(z) := (a_i^\top z - b_i)/\|a_i\|_2$ be the normalized slack of constraint $i$ at $z$, and $\mathcal L_K(z) := \sum_k \ell(x^k - z)$.

\begin{definition}[Two-stage $\PIO$]\label{def:twostage}
Let $\hat z$ solve \eqref{eq:PIO}. \emph{Rounding:} scan the rows of $A$ by increasing $s_i(\hat z)$, keep each row linearly independent of those already kept, stop at $n$ rows, and if $n$ rows are kept let $v$ be the point at which they are active. \emph{Guard:} if $v$ exists, $v \in \Omega$, and $\mathcal L_K(v) - \mathcal L_K(\hat z) \le \tau_K$, the estimate is $\tilde z := v$; otherwise $\tilde z := \hat z$. The reported set is $\Theta^*(\tilde z)$.
\end{definition}

The rows active at $\hat z$ have zero slack and are scanned first, so a feasible $v$ is a vertex of the smallest face of $\Omega$ containing $\hat z$; the rounding costs one sort, at most $m$ independence tests, and one linear system. We use $\tau_K := \log K \cdot \min_z \mathcal L_K(z)/K$, the logarithm of the sample size times the average loss of the unconstrained fit (the centroid for the squared loss, the coordinatewise medians for the absolute loss), a penalty of the type of the Bayesian information criterion, the same for every face.

\begin{theorem}[Exact recovery of the optimal vertex]\label{thm:twostage}
In the linear case with all rows $a_i \ne 0$ and $\Theta = \{\theta : \|\theta\|_2 = 1\}$, let Assumption~\ref{assump:DGP} hold, let $\hat z_K$ be any measurable selection of minimizers of \eqref{eq:PIO}, with rounding $v_K$ and estimate $\tilde z_K$, and let $r := \min_{i \notin I(z_0)} s_i(z_0)$ ($r := \infty$ if every constraint is active at $z_0$).
\begin{enumerate}[(i)]
\item If $z_0$ is a vertex and $\|\hat z_K - z_0\|_2 < r/2$, then $v_K = z_0$. Hence $\mathbb{P}(v_K \ne z_0) \le 16\sigma^2/(Kr^2)$ for every $K$ under the squared loss, and $v_K = z_0$ for all large $K$ almost surely under the absolute loss if $\Omega$ is bounded.
\item If $z_0$ is a vertex, then under the squared loss $\mathbb{P}(\tilde z_K \ne z_0) \le 16\sigma^2/(Kr^2) + \mathbb{P}(K\|\bar x_K - z_0\|_2^2 > \tau_K)$, and, for the $\tau_K$ above, both terms tend to zero.
\item If $z_0$ is not a vertex and $\tau_K/K \to 0$ almost surely, as for the $\tau_K$ above, then $\tilde z_K = \hat z_K$ for all large $K$ almost surely, so $\tilde z_K \to z_0$ almost surely (under the absolute loss if $\Omega$ is bounded).
\end{enumerate}
Whenever $\tilde z_K = z_0$, the reported set is the whole set $\Theta^*(z_0) = \Kc(z_0) \cap \Theta$.
\end{theorem}

Under squared loss, the unguarded rounding candidate recovers the true vertex with probability $1 - O(K^{-1})$. The guarded estimator recovers it with probability tending to one.

\section{Computational Study} \label{Section:NumericalEx}

Section~\ref{subsec:Ksweep} tests how $\PIO$ and the classical formulations scale with $K$ and $n$ and how accurate their decisions and objectives are. Section~\ref{subsec:accuracy} varies the noise level and the number of constraints binding at the true optimum, including an optimum in the relative interior of a face, with the short records typical of dietary data. The main runs of Section~\ref{subsec:Ksweep} were one resumable job, restarted several times between 19 and 29 October 2026, on a shared Linux server (Intel Core i9-13900K, 24 cores, 125~GB of memory) with Gurobi~13.0.3 and four threads pinned to four cores; Sections~\ref{subsec:accuracy} and \ref{Section:Application} ran on a two-core virtual machine with Gurobi~13.0.3 (times there are indicative only). Times are wall-clock seconds and include model construction.

\subsection{Scaling with the Number of Observations} \label{subsec:Ksweep}

We generate random forward instances of two types. \emph{Linear:} $\Omega = \{x \in [-10,10]^n : \tilde A x \leq \tilde b\}$ with $\tilde A \in \mathbb{Z}^{n \times n}$ having entries uniform in $\{-10,\ldots,10\}$, each set to zero with probability $1/2$, and $\tilde b = \tilde A x^0 + u$ with $x^0$ uniform in $[-5,5]^n$ and $u$ uniform in $[1, 100]^n$; the box bounds are explicit rows, giving $m = 3n$ constraints. The objective is $\theta_0^\top x$ with $\theta_0$ a normalized standard Gaussian vector, so that the true optimum $z_0$ is a unique vertex with probability one (Theorem~\ref{thm:generic-uniqueness}). \emph{Convex quadratic:} the same $\Omega$ with basis functions $\phi_j(x) = \tfrac12 (x - c_j)^\top Q_j (x - c_j)$, $j = 1,\ldots,p$, for random positive semidefinite $Q_j$ of rank $n/2$ and random centers $c_j \in [-5,5]^n$, with $p = 3$ and $\theta_0$ drawn from the simplex; here the gradient matrix $\Phi(z)$ varies with $z$ and Proposition~\ref{prop:linear-structure} does not apply. For each instance the forward problem is solved to obtain $z_0$, and $K$ observations $x^k = z_0 + \xi^k$ with $\xi^k \sim \mathcal{N}(0, \sigma^2 I)$ are generated, with per-coordinate standard deviation $\sigma = 1$ (5\% of the side of the box; the $\sigma^2$ of Assumption~\ref{assump:DGP}, $\mathbb{E}\|\xi^k\|_2^2$, is then $n$). We vary $n \in \{10, 20, 50\}$ and $K \in \{10, 10^2, 10^3, 10^4\}$, with 20 instances per cell. A second experiment fixes $K = 1000$ and grows the dimension, $n \in \{10, 50, 100, 500, 1000\}$ with $m = 3n$, five instances per $n$.

\paragraph{\bf Models.}
$\ILO$-KKT and $\PIO$-KKT are \eqref{eq:ICO} (\eqref{eq:ILO} on linear instances) and \eqref{eq:PIO} with the absolute loss, complementarity handled by indicator constraints ($v_i = 0 \Rightarrow \lambda_i = 0$, $v_i = 1 \Rightarrow g_i(z) = 0$) with $Km$ and $m$ binary variables respectively, and the normalization $\mathbf{1}^\top \lambda^k = 1$ for each imputed optimum in the linear case (since $\Omega$ has interior points here, it excludes $\theta = 0$ and gives the same feasible decisions as the unit sphere: the set $\bigcup_i F_i$ for $\PIO$, and for $\ILO$ because all imputed optima lie on the optimal face of the shared $\theta$) and $\mathbf{1}^\top\theta = 1$, $\theta \geq 0$ in the quadratic case (where stationarity is bilinear in $(z,\theta)$ and is passed to the solver as a nonconvex quadratic constraint). In $\PIO$-KKT the loss enters as the piecewise-linear objective of Proposition~\ref{prop:aggregation}, so the model contains no per-observation variables. For linear instances we also report $\ILO$-shared, the exact reformulation of $\ILO$ with one multiplier vector and $m$ binaries (Theorem~\ref{thm:complexity}); $\ILO$-enum, the face enumeration of \citet{babier2021ensemble} (Proposition~\ref{prop:linear-structure}(ii)); $\PIO$-enum, its tied analogue; and $\PIO$-relax, one convex program over $\Omega$ with the boundary check of Proposition~\ref{prop:linear-structure}(iv) and the enumeration as fallback, solved with the absolute loss (a linear program) and with the squared loss (a quadratic program). On quadratic instances both KKT formulations are also solved with the squared loss. As convex baselines we solve, on the linear instances, $\ILO$-residual \citep{keshavarz2011imputing}, which minimizes the absolute stationarity and complementarity residuals of each observation, and $\ILO$-gap \citep{bertsimas2015data}, which minimizes the absolute duality gap of each observation: linear programs with one multiplier vector per observation, normalized as above, that return a cost vector but no decision. A time limit of 600~s is imposed on the mixed-integer formulations and the convex baselines, and a cell is stopped once its first five instances have all reached it, in which case the model is not run at larger $K$; $\PIO$-relax and the enumerations are run to completion, except that $\ILO$-enum is limited to one hour per instance and an enumeration cell is stopped once its first five instances take more than 2.5~h in total; both bind only at $n = 50$, $K = 10^4$, for $\ILO$-enum. Other users' jobs shared the server (median one-minute load 19 overall and 1 to 54 per cell, on 32 logical processors; the convex baselines ran on 1--2 November 2026 at median load 3.5), so the times are noisy; the main comparisons below rest on differences of a factor of ten or more, and the closer ones among the classical formulations agree with Gurobi's deterministic work units. Optimal values agree across formulations to a relative $10^{-4}$ on every instance that two of them solved.

\paragraph{\bf Metrics.}
We record the solve time; the $\ell_1$ error to $z_0$ of the pooled optimum $\hat z$, of the two-stage estimate $\tilde z$, of the mean imputed $\ILO$ optimum, and of the recover-then-optimize decision $\arg\min_{x \in \Omega} \hat\theta^\top x$ built from the returned parameter $\hat\theta$; whether $\hat\theta$ lies in the true cone $\Theta^*(z_0)$; and the cosine between $\hat\theta$ and $\theta_0$, against the best and typical extreme rays of $\Theta^*(z_0)$ (the largest and the median cosine between an extreme ray and $\theta_0$).

\paragraph{\bf Results: time.}
Table~\ref{Table:ksweep} and Figure~\ref{fig:ksweep}(a)--(c) give median times; three findings stand out.

First, on linear instances one convex program solves $\PIO$: $\PIO$-relax never needed its fallback (0 of 530 solves, including the dimension study). With the squared loss it takes 3 to 8~ms at every $n$ and $K$, because the data enter only through $\bar x$. With the absolute loss each of the $n$ piecewise-linear terms has up to $K$ breakpoints, and the time grows to $0.13$, $0.89$, and $8.0$~s at $K = 10^4$ for $n = 10$, $20$, and $50$ ($\PIO$-KKT: $0.29$, $1.2$, and $7.3$~s).

Second, every exact $\ILO$ formulation slows down as $K$ grows. $\ILO$-shared, with as few binaries as $\PIO$-KKT, is the fastest of them at small $K$ ($0.07$ to $0.49$~s at $K = 10$), but at $K = 10^3$ it takes $12$, $28$, and $147$~s (for $n = 10$, $20$, and $50$, as below) against $0.006$ to $0.12$~s for $\PIO$-relax, and at $K = 10^4$ it reaches the limit on 5 of 20, 14 of 20, and 5 of 5 instances. $\ILO$-KKT, with $Km$ binaries, reaches the limit on 14 of 20 instances at $n = 20$, $K = 100$, on all five at $K = 10^3$ for $n = 10$ and $20$, and on all five at $n = 50$ already at $K = 100$. $\ILO$-enum grows linearly in $K$, to $313$ and $1{,}552$~s at $K = 10^4$ for $n = 10$ and $20$ (two thirds of it model construction), and at $n = 50$ all five instances reach its one-hour limit; its tied analogue $\PIO$-enum takes $3.3$, $45$, and $1{,}041$~s. The convex baselines avoid binaries but not the growth in $K$: at $K = 10^4$, $\ILO$-gap and $\ILO$-residual take $29$ and $61$~s at $n = 10$ and $194$ and $286$~s at $n = 20$, against $0.13$ and $0.89$~s for $\PIO$-relax, and at $n = 50$ all five instances of each reach the limit, although they ran under lighter load.

Third, on quadratic instances, where stationarity is bilinear and no enumeration is available, $\ILO$-KKT reaches the limit on every instance with either loss already at $K = 10$. $\PIO$-KKT with the absolute loss solves every instance at $n = 10$ for every $K$ (median $0.34$ to $34$~s) and at $n = 20$ for $K \le 10^3$ ($6.4$ to $79$~s), but its time grows with $K$, and at $n = 20$, $K = 10^4$, 16 of 20 instances reach the limit. With the squared loss the time falls as $K$ grows: from $51$ to $3.6$~s at $n = 10$, and at $n = 20$ from 15 of 20 instances at the limit to none (median $42$~s at $K = 10^4$). At $n = 50$, 14 of 20 absolute-loss instances reach the limit at $K = 10$ and all five at $K = 100$, and with the squared loss all five instances reach it at every $K$ from $10$ to $10^4$ (the larger $K$ in a separate run).

At fixed $K = 1000$ (Figure~\ref{fig:ksweep}(c)), as $n$ grows from $10$ to $1000$ ($m = 3n$), the median time of $\PIO$-relax is $0.003$ to $2.5$~s with the quadratic program and $0.005$ to $126$~s with the linear program (dual simplex, for exact boundary solutions), and $\PIO$-KKT solves every instance (medians $0.03$ to $69$~s). %
We ran $\ILO$ only for $n \le 50$. At $n = 50$, $\ILO$-KKT times out on all five instances, while $\ILO$-shared and $\ILO$-enum take $87$~s and $1{,}371$~s, respectively.

\paragraph{\bf Results: objective.}
No formulation identifies $\theta_0$, as Corollary~\ref{cor:non-uniqueness} predicts (Figure~\ref{fig:ksweep-theta}). The face normal $\hat\theta$ of either face enumeration lies in the true cone $\Theta^*(z_0)$ in $95$--$100\%$ of instances in every cell that finished (for $\PIO$, as Theorem~\ref{thm:cone-recovery} predicts), but its median cosine with $\theta_0$ is only $0.07$--$0.23$ for $\PIO$, about that of a typical extreme ray, and $-0.01$ to $0.19$ for $\ILO$, with no gain in $K$; the two select the same face in $0$--$20\%$ of instances. The baselines' rays are not face normals: median cosine $0.10$--$0.30$ ($\ILO$-residual) and $0.28$--$0.74$ ($\ILO$-gap, in the true cone in $75$--$100\%$ of instances for $K \ge 10^3$).

\paragraph{\bf Results: decision.}
On the decision $\PIO$ and $\ILO$ are very different (Figure~\ref{fig:ksweep}(d)). Recovering $\hat\theta$ and re-optimizing is poor for both models (median $\ell_1$ error $66$ to $81$ at $n = 10$ and $361$ to $419$ at $n = 50$ for every exact formulation): the whole recovered face is optimal for $\hat\theta$, so the re-optimized decision is whichever of its vertices the solver returns. The ray of $\ILO$-residual is about as poor ($55$ to $74$ and $343$ to $395$); that of $\ILO$-gap gives median errors of $4.8$ to $16.0$ at $n = 10$, $17.5$ to $33.8$ at $n = 20$, and $49.9$ to $113.7$ at $n = 50$, and $z_0$ itself in at most $20\%$ of instances. Averaging the $K$ imputed $\ILO$ optima is much better ($3.6$ to $3.8$ at $n=10$, $7.0$ to $8.0$ at $n=20$, $17.8$ to $18.9$ at $n = 50$) but barely improves with $K$: $z_0$ is a vertex, each per-observation projection onto the selected face is confined to one side of the neighboring constraints, and the resulting bias of order $\sigma$ along the face does not average out. The pooled optimum $\hat z$ of $\PIO$ (from $\PIO$-relax; other formulations give medians within $0.1$) has median error $1.5 \to 0.06$ ($n=10$), $3.3 \to 0.12$ ($n=20$), and $9.0 \to 0.31$ ($n = 50$) as $K$ grows from $10$ to $10^4$ with the absolute loss, and slightly smaller errors with the squared loss. For the squared loss, Theorem~\ref{thm:consistency}(i) bounds the mean $\ell_2$ error by $2(\mathbb{E}\|\xi^k\|_2^2/K)^{1/2} = 2\sqrt{n/K}$ here; the observed mean is $0.30$ to $0.39$ times this bound in every cell, so the bound has the right rate and is about three times too loose. On quadratic instances the error falls in the same way ($0.84 \to 0.04$ at $n = 10$, squared loss).

\paragraph{\bf Results: two-stage PIO.}
With the squared loss, two-stage $\PIO$ recovers $z_0$ exactly in 75\%, 85\%, 100\%, and 100\% of instances at $n = 10$ as $K$ grows from $10$ to $10^4$, 50--100\% at $n = 20$, and 0--90\% at $n = 50$; with the absolute loss the shares are 65--100\%, 40--100\%, and 0--90\%. $\PIO$ alone never recovers $z_0$. The guard accepts a wrong vertex in at most 20\% of instances per cell and in none at $K = 10^4$; where no vertex is recovered ($n = 50$, $K = 10$), the median error is 2--5\% above $\PIO$'s. The rounding and guard take a median of at most 29~ms per cell here (0.5 to 0.6~s at $n = 1000$). At $K = 1000$ in the dimension study two-stage $\PIO$ recovers $z_0$ in five of five instances at $n = 10$ and in none at $n \ge 500$.

\begin{table}[tb]
\footnotesize
\begin{center}\renewcommand{\arraystretch}{1.1}
\caption{Median total time (s, model construction included) at $n=20$ ($m=60$) over 20 instances (5 where the stopping rule of Section~\ref{subsec:Ksweep} ended a cell: linear $\ILO$-KKT at $K = 10^3$ and quadratic $\ILO$-KKT). Loss: absolute (abs.) or squared (sq.). TL: more than half of the instances reached the $600$~s limit (elsewhere at most half did: 8 and 3 of 20 in the quadratic squared-loss $\PIO$-KKT cells with medians 333 and 106, and 1 of 20 for $\ILO$-residual at $K = 10^4$). --: not run (Section~\ref{subsec:Ksweep}); n/a: requires a linear forward problem (the convex baselines were run on linear instances only).}
\label{Table:ksweep}
\setlength{\tabcolsep}{4pt}
\begin{tabular}{lll|cccc|cccc}
\hline\hline
& & & \multicolumn{4}{c|}{Linear, $n=20$} & \multicolumn{4}{c}{Quadratic, $n=20$, $p=3$} \\
Model & Loss & Binaries & $K{=}10$ & $10^2$ & $10^3$ & $10^4$ & $K{=}10$ & $10^2$ & $10^3$ & $10^4$ \\
\hline
$\PIO$-relax & abs. & 0 & $<0.01$ & $<0.01$ & 0.02 & 0.89 & \multicolumn{4}{c}{n/a} \\
$\PIO$-relax & sq. & 0 & $<0.01$ & $<0.01$ & $<0.01$ & $<0.01$ & \multicolumn{4}{c}{n/a} \\
$\PIO$-KKT & abs. & $m$ & 0.03 & 0.04 & 0.09 & 1.2 & 6.4 & 10.8 & 79.1 & TL \\
$\PIO$-KKT & sq. & $m$ & -- & -- & -- & -- & TL & 333 & 106 & 41.8 \\
$\PIO$-enum & abs. & 0 & 0.50 & 0.53 & 1.4 & 44.9 & \multicolumn{4}{c}{n/a} \\
$\ILO$-KKT & abs. & $Km$ & 1.1 & TL & TL & -- & TL & -- & -- & -- \\
$\ILO$-KKT & sq. & $Km$ & -- & -- & -- & -- & TL & -- & -- & -- \\
$\ILO$-shared & abs. & $m$ & 0.14 & 1.5 & 28.3 & TL & \multicolumn{4}{c}{n/a} \\
$\ILO$-enum & abs. & 0 & 1.1 & 11.9 & 138 & 1,552 & \multicolumn{4}{c}{n/a} \\
$\ILO$-residual & abs. & 0 & 0.02 & 0.25 & 5.5 & 286 & \multicolumn{4}{c}{n/a} \\
$\ILO$-gap & abs. & 0 & 0.02 & 0.23 & 3.4 & 194 & \multicolumn{4}{c}{n/a} \\
\hline\hline
\end{tabular}
\end{center}
\end{table}

\begin{figure}[!tbp]
    \centering
    \includegraphics[width=\textwidth]{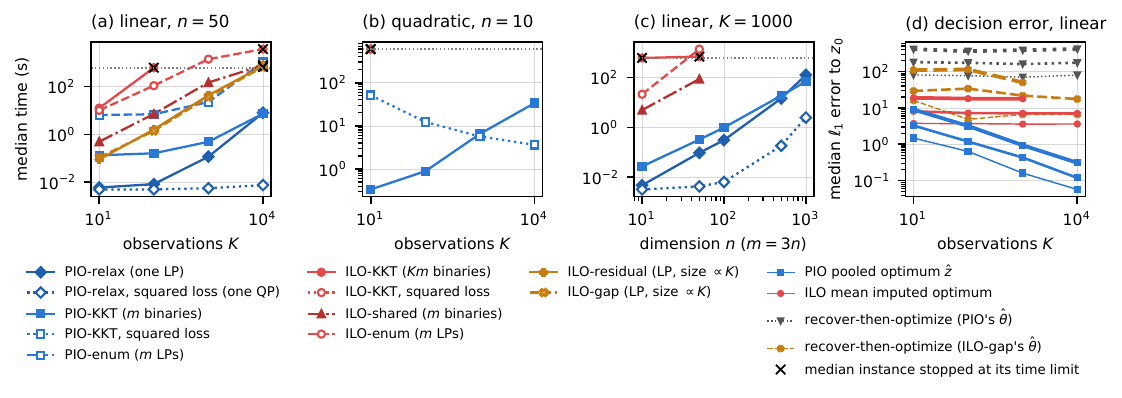}
    \caption{Solve time and decision error. (a)--(c) Median total time against $K$ for linear instances with $n = 50$ and quadratic instances with $n = 10$ (20 instances per cell; 5 where the stopping rule of Section~\ref{subsec:Ksweep} ended a cell), and against $n$ ($m = 3n$) at $K = 1000$ (five per $n$). Dotted line: the $600$~s limit; a cross marks a cell in which more than half of the instances stopped at a limit (for $\ILO$-enum at $n = 50$, $K = 10^4$, its one-hour limit). Absolute loss unless labeled squared loss; on quadratic instances $\ILO$-KKT was run only at $K = 10$. (d) Linear instances, $n \in \{10, 20, 50\}$ (thicker lines for larger $n$): median $\ell_1$ error to $z_0$ of $\PIO$'s $\hat z$ (absolute loss), of the mean imputed $\ILO$ optimum, and of the recover-then-optimize decisions built from $\PIO$'s $\hat\theta$ ($\ILO$'s is of the same size) and from $\ILO$-gap's.}
    \label{fig:ksweep}
\end{figure}

\subsection{Accuracy Across Noise Levels and Binding Structures} \label{subsec:accuracy}

Random linear instances in $\mathbb{R}^{10}$ are generated with $m$ structural constraints, $m$ uniform in $\{10,\ldots,30\}$, $A \in \mathbb{Z}^{m \times n}$ with entries uniform in $\{-10,\ldots,9\}$ and half of them zero, $b \in \{-5,\ldots,4\}^m$, rejecting instances whose feasible region does not meet $[0,10]^{10}$, plus the box $[-10,10]^{10}$ as $2n$ explicit rows so that every constraint active at $z_0$ is an explicit row. The true optimum is generated in two ways. In the \emph{vertex scenario}, $\theta_0$ is drawn uniformly from $[-10,10]^n$ and $z_0$ is the (almost surely unique) optimal vertex. In the \emph{face scenario}, a point $x^0$ is drawn uniformly from the box and projected onto the $\ell_1$-nearest rationalizable point of $\Omega$ by solving \eqref{eq:PIO} with the single observation $x^0$; the projection is $z_0$ and the returned parameter is $\theta_0$, in every generated instance the normal of a single constraint active at $z_0$. Instances are rejection-sampled so that exactly $\kappa \in \{1, 5, 10\}$ structural constraints are active at $z_0$ to within $0.1$ (box rows are not counted). In the vertex scenario $z_0$ is a vertex whatever $\kappa$, and $\kappa$ measures how many of its $n$ active constraints are structural; in the face scenario $z_0$ lies in the relative interior of a face of codimension $\kappa$ in 86--89\% of instances (a vertex when $\kappa = n = 10$, but with $\theta_0$ a single active normal, on the boundary of the cone) and of codimension $\kappa \pm 1$ in the rest, a situation that Theorem~\ref{thm:generic-uniqueness} excludes for generic $\theta_0$ but that arises when the decision-maker is indifferent along a face. Observations are $x^k = z_0 + \xi^k$, $\xi^k \sim \mathcal{N}(0,\sigma^2 I)$, with $K$ drawn uniformly from $\{4,\ldots,7\}$ and per-coordinate noise $\sigma \in \{0.02, 0.5, 5\}$ on a feasible region inside the box of side $20$. For each scenario and $\kappa$, 100 instances are generated and solved at every $\sigma$ with new noise; $\PIO$ and $\ILO$ are solved by face enumeration with the absolute loss, in $0.1$ and $0.6$~s per instance.

\paragraph{\bf Metrics.}
We report (a) the $\ell_1$ distance to $z_0$ of $\hat z$, of the two-stage estimate $\tilde z$, of the mean imputed $\ILO$ optimum, and of the recover-then-optimize decision; (b) the \emph{validity} of the reported cone, the share in which every constraint active at $\hat z$ is active at $z_0$, so that $\Theta^*(\hat z) \subseteq \Theta^*(z_0)$ (Theorem~\ref{thm:cone-recovery}(i)), and, for both models, the share in which the selected face passes through $z_0$, so that its normal $\hat\theta$ lies in $\Theta^*(z_0)$; (c) the \emph{coverage} of the reported cone, the share in which $\theta_0 \in \Theta^*(\hat z)$, tested by a feasibility linear program on the constraints active at $\hat z$ (for $\ILO$, at the mean imputed optimum); and (d) the share in which $\hat z = z_0$ or $\tilde z = z_0$.

\paragraph{\bf Results.}
Table~\ref{Table:accuracy} collects the results. On the decision, $\hat z$ is closer to $z_0$ than the mean imputed $\ILO$ optimum in every reported cell of the vertex scenario, by a factor of $1.4$--$2.2$; in the face scenario, %
the two are comparable (the mean imputed optimum is 11--14\% closer at $\kappa = 1$, $\hat z$ is $1.5$--$1.6$ times closer at $\kappa = 10$). 
Re-optimization yields median decision errors of $35$--$66$ across both models, with results sensitive to solver tie-breaking when the fitted cost admits multiple optima. 
On the objective, in the vertex scenario the reported cone is valid in $82$--$100\%$ of instances at $\sigma \le 0.5$ and in $33$--$52\%$ at $\sigma = 5$. By the second criterion of (b), the face normals of $\PIO$ and $\ILO$ are valid at similar rates at $\sigma \le 0.5$, and at $\sigma = 5$ in $54$--$92\%$ of instances for $\PIO$ against $88$--$90\%$ for $\ILO$. Coverage is essentially zero for both models: when $\hat z$ lies on a face through the vertex, it reports a strict subcone of the $n$-dimensional $\Theta^*(z_0)$, which contains a generic $\theta_0$ only when $\hat z = z_0$ exactly ($0$--$6\%$ of instances with $K \le 7$; for the squared loss, the discussion after Theorem~\ref{thm:cone-recovery} gives the limiting probability of this event). Two-stage $\PIO$ (with the absolute loss, as in Table~\ref{Table:accuracy}) lifts this limit: it recovers $z_0$ exactly in $98$--$100\%$ of vertex-scenario instances at $\sigma = 0.02$, $47$--$77\%$ at $\sigma = 0.5$, and $11$--$15\%$ at $\sigma = 5$, with a lower median error than $\PIO$ in every vertex cell; with the squared loss, the case that Theorem~\ref{thm:twostage}(ii) covers, the shares are $94$--$98\%$, $49$--$80\%$, and $10$--$14\%$. The price is validity at higher noise: with the absolute loss the guard accepts a wrong vertex, whose cone never lies in $\Theta^*(z_0)$, in $9$--$25\%$ of these instances at $\sigma = 0.5$ and $40$--$46\%$ at $\sigma = 5$, so the reported set is valid less often than $\PIO$'s cone ($69\%$ against $82\%$ at $\kappa = 10$, $\sigma = 0.5$). Most other misses are roundings without a feasible vertex; the guard rejects at most $5\%$. In the face scenario the reported cone is valid at $\sigma \le 0.5$ in $69$--$100\%$ of instances and, at all three noise levels, covers $\theta_0$ in $54$--$98\%$ for $\PIO$ against $14$--$95\%$ for $\ILO$, because $\theta_0$ is the normal of a constraint active at $z_0$. %
$\PIO$-relax (median $1.3$~ms) uses its fallback in at most 3\% of runs per cell, except in the $\kappa = 1$ face scenario (up to 49\%), where sample medians often lie in $\Omega$.

\begin{table}[tb]
\footnotesize
\begin{center}\renewcommand{\arraystretch}{1.0}
\caption{Accuracy at $n = 10$ with $K \in \{4,\ldots,7\}$ observations and the absolute loss, 100 instances per cell ($\kappa$: structural constraints active at $z_0$; $\sigma$: noise level). Decision error: median $\ell_1$ distance to $z_0$ of the pooled optimum $\hat z$, of the two-stage estimate $\tilde z$ (2-st.), of the mean imputed $\ILO$ optimum, and of the recover-then-optimize decision (RTO, shown for $\ILO$; $\PIO$'s median differs by less than 7.2). Valid, covers, and exact as defined under Metrics ($\PIO$ cone: $\Theta^*(\hat z) \subseteq \Theta^*(z_0)$; $\hat\theta$: the selected face passes through $z_0$; exact: $\ell_1$ distance below $10^{-6}$).}
\label{Table:accuracy}
\setlength{\tabcolsep}{3.5pt}\begin{tabular}{llr|rrrr|rrr|rr|rr}
\hline\hline
& & & \multicolumn{4}{c|}{Decision error} & \multicolumn{3}{c|}{Valid (\%)} & \multicolumn{2}{c|}{Covers $\theta_0$ (\%)} & \multicolumn{2}{c}{Exact (\%)} \\
Scenario & $\kappa$ & $\sigma$ & $\PIO$ $\hat z$ & 2-st. & $\ILO$ mean & RTO & $\PIO$ cone & $\PIO$ $\hat\theta$ & $\ILO$ $\hat\theta$ & $\PIO$ & $\ILO$ & $\PIO$ & 2-st. \\
\hline
vertex & 1 & 0.02 & 0.04 & 0.00 & 0.08 & 55 & 100 & 100 & 100 & 0 & 0 & 0 & 98 \\
vertex & 1 & 0.5 & 0.90 & 0.00 & 1.94 & 53 & 97 & 97 & 100 & 2 & 0 & 2 & 77 \\
vertex & 1 & 5 & 9.17 & 8.17 & 18.37 & 59 & 44 & 54 & 90 & 0 & 0 & 0 & 12 \\
vertex & 5 & 0.02 & 0.05 & 0.00 & 0.08 & 65 & 100 & 100 & 100 & 0 & 0 & 0 & 100 \\
vertex & 5 & 0.5 & 1.14 & 0.00 & 2.07 & 59 & 92 & 98 & 99 & 0 & 0 & 0 & 69 \\
vertex & 5 & 5 & 10.32 & 9.40 & 16.96 & 60 & 52 & 92 & 89 & 0 & 0 & 0 & 11 \\
vertex & 10 & 0.02 & 0.04 & 0.00 & 0.07 & 42 & 99 & 99 & 100 & 1 & 0 & 1 & 99 \\
vertex & 10 & 0.5 & 0.99 & 0.23 & 1.41 & 41 & 82 & 97 & 99 & 3 & 0 & 3 & 47 \\
vertex & 10 & 5 & 6.75 & 6.28 & 11.19 & 42 & 33 & 81 & 88 & 6 & 0 & 6 & 15 \\
face & 1 & 0.02 & 0.08 & 0.08 & 0.07 & 65 & 100 & 100 & 100 & 98 & 95 & 0 & 0 \\
face & 1 & 0.5 & 2.12 & 2.12 & 1.89 & 65 & 69 & 75 & 67 & 73 & 64 & 0 & 0 \\
face & 1 & 5 & 17.56 & 18.03 & 15.19 & 65 & 3 & 32 & 25 & 55 & 24 & 0 & 0 \\
face & 5 & 0.02 & 0.07 & 0.07 & 0.08 & 60 & 100 & 100 & 100 & 54 & 21 & 0 & 0 \\
face & 5 & 0.5 & 1.58 & 1.58 & 1.89 & 60 & 86 & 98 & 93 & 57 & 19 & 0 & 0 \\
face & 5 & 5 & 14.05 & 14.97 & 14.99 & 63 & 18 & 71 & 78 & 59 & 14 & 0 & 0 \\
face & 10 & 0.02 & 0.04 & 0.00 & 0.07 & 45 & 98 & 100 & 100 & 60 & 22 & 3 & 88 \\
face & 10 & 0.5 & 1.03 & 0.76 & 1.61 & 46 & 83 & 98 & 100 & 68 & 20 & 3 & 38 \\
face & 10 & 5 & 8.16 & 8.27 & 12.43 & 47 & 24 & 85 & 94 & 70 & 27 & 4 & 11 \\
\hline\hline
\end{tabular}
\end{center}
\end{table}

\section{Application: Diet Recommendations} \label{Section:Application}

We illustrate $\PIO$ on dietary recommendations for women with hypertension. The Dietary Approaches to Stop Hypertension (DASH) eating plan sets nutrient goals \citep{sacks2001effects, liese2009adherence}, but few adults with hypertension follow it \citep{mellen2008deteriorating}. $\PIO$ turns daily intake records into the diet closest to the record that meets the goals, together with the goals that bind. Because the recorded diets do not meet the goals, $\PIO$ acts here as a projection (Proposition~\ref{prop:linear-structure}(iii)), the operation of individual diet modeling \citep{maillot2010individual}: the application shows what the model returns on real data and how fast, not the recovery of a latent diet.

\subsection{Data and Feasible Set} \label{sec:app_data}

\paragraph{\bf Observations.}
We use the 2015--2016 cycle of the National Health and Nutrition Examination Survey (NHANES), whose dietary interview covers 9,544 persons, with up to two 24-hour recalls each \citep{CDC_2020}. We focus on women aged 51 to 75 who report hypertension and either prediabetes or diabetes: 243 women with 460 reliable recall days (217 with two days). Each recall day is one observation $x \in \mathbb{R}^{38}$, the servings of 38 food groups obtained by grouping the food codes that appear in the recalls by their first two digits, named as in the USDA coding scheme (Section~\ref{Appendix:diet_data} lists the excluded codes). Every recall day misses at least one of the bounds below: the potassium, magnesium, and fiber goals on 82--90\% of days, and the sodium limit on 62\% (Table~\ref{Table:SubsetNutrients}).

\paragraph{\bf Feasible region.}
The nutrient content of one serving of each group is the gram-weighted nutrient density of all NHANES 2015--2016 recall items in the group, times the serving size, with the nutrient values of the USDA Food and Nutrient Database for Dietary Studies \citep{usda_2019}. With this $12 \times 38$ matrix $N$, the nutrient totals implied by a day's food-group servings correlate with the recalled totals at 0.67--0.92 across nutrients. The feasible set is
\[
\Omega = \{x \in \mathbb{R}^{38} : \underline{b} \le N x \le \overline{b},\ 0 \le x \le c\}.
\]
The bounds $\underline{b}, \overline{b}$ are the nutrient goals of the DASH trials \citep{nhlbi2006dash}, expressed for the 1,600-kcal plan, the level given for sedentary women over 50 \citep{nhlbi2015brief}: total fat, saturated fat, and carbohydrate at most 27\%, 6\%, and 55\% of energy, protein at least 18\%; potassium, calcium, magnesium, and fiber at least their 2,100-kcal goals scaled by $1{,}600/2{,}100$; sodium at most 2,300~mg and cholesterol at most 150~mg; and energy within 10\% of 1,600~kcal. The caps $c$ are acceptability constraints as in individual diet modeling \citep{maillot2010individual}: $c_g$ is the 95th percentile of the servings of group $g$ on the cohort's recall days that contain it, and $c_g = 0$ for the two groups no woman ate and for formulated nutrition, energy, and sports drinks, which are not part of the DASH plan. This gives $n = 38$ and $m = 12 + 76 = 88$ constraints. Because three groups are fixed at zero, $\Omega$ has no interior; the results for linear $\PIO$ used below need no constraint qualification. The bounds are in Table~\ref{Table:SubsetNutrients}.

\subsection{Recommendations} \label{sec:app_results}

Each woman's observations are her one or two recall days, and we use the squared loss, so that the data enter through the mean of her days (Proposition~\ref{prop:aggregation}). No woman's mean day lies in $\Omega$, so by Proposition~\ref{prop:linear-structure}(iii) every recommendation is the Euclidean projection of her mean day onto $\Omega$, one quadratic program: the 243 problems take a median of $3$~ms each (maximum $0.1$~s), construction included, and, as a check of the implementation, agree with the face enumeration (88 programs, median $1.8$~s) to a relative difference below $10^{-9}$.

Across women, the recommendation removes a median of 53\% of the recorded servings (interquartile range 44--59\%) and adds a median of 7.8 servings. On average it removes baked goods and crackers (0.9 servings) and grain mixtures such as pizza and sandwiches (0.7), and it adds vegetable mixtures (1.5), citrus fruit (1.4), fish (1.2), legumes (0.9), seeds (0.7), and milk and yogurt (0.6), mostly groups the woman did not record (Figure~\ref{Fig:IO_optimalfoodintakes}(a)). Besides the three groups fixed at zero, it sets eleven groups that some women ate, among them sugars and sweets, fats and oils, and sausages and lunchmeats, to zero for at least 95\% of women; the Euclidean norm weighs one serving of every group alike, and weighting the groups is a choice for the dietitian. 
The saturated-fat limit and calcium goal are active in every projected diet. The carbohydrate limit and the protein and magnesium goals are active for 97--99\% of women. The potassium goal is active for 53\%, the cholesterol and sodium limits for 45\% and 34\%, respectively, and the fiber goal for 22\%.
A median of seven nutrient bounds bind (range 5--8). In nutrient terms, computed with the food-group model $N$ (Figure~\ref{Fig:IO_optimalfoodintakes}(b)), median saturated fat falls from 21.2 to 10.7~g and sodium from 2,887 to 2,196~mg, and potassium, calcium, and magnesium rise from 2,024, 626, and 212~mg to their goals; total sugars, which DASH does not bound, rise from 87 to 122~g with the added fruit and dairy. With the 1,600-kcal ranges of \citet{ahmadi2020opensource}, one-sided except for energy, the fiber and sodium bounds bind for every woman and the median share removed is 49\%. For the 217 women with two recall days we also solve $\PIO$ on each day alone. The two single-day recommendations bind the same nutrient bounds for only 14\% of these women (median Jaccard similarity 0.75), so the bounds that bind after one day depend on which day happened to be recorded.

\begin{figure}[htb]
\begin{center}
\includegraphics[width=0.86\linewidth]{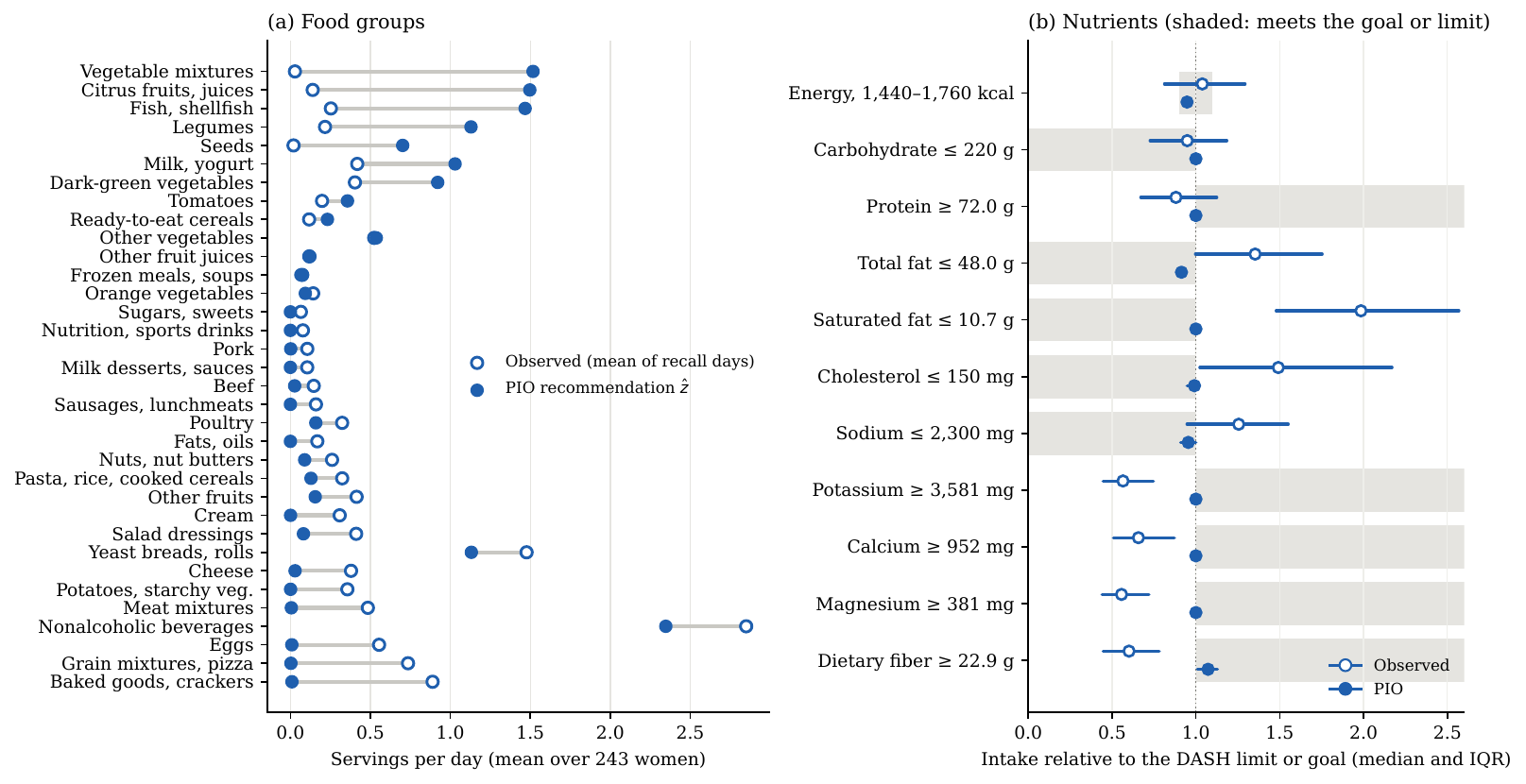}
\caption{Observed diets and $\PIO$ recommendations for 243 women aged 51--75 with hypertension and prediabetes or diabetes (NHANES 2015--2016). (a)~Mean daily servings for the 34 food groups with a mean of at least 0.05 servings in either diet: observed (mean of each woman's recall days, open) and recommended $\hat z$ (filled). (b)~Daily nutrient content $Nx$ of the same diets relative to the DASH limit or goal, median and interquartile range across women; the shaded side meets the limit or goal.} \label{Fig:IO_optimalfoodintakes}
\end{center}
\end{figure}

\section{Conclusion}
\label{sec:conclusion}

When one decision-maker facing a fixed problem is observed many times, the optimal decision is generically a unique vertex whose cone of rationalizing objectives does not shrink with more observations. \PIOname{} therefore targets the decision and its cone, with a model whose size does not grow with $K$, a finite-sample bound for its decision, a reported cone that eventually lies inside the true one, and recovery of the optimal vertex by rounding.

The limitations of this work and potential for future work include the following. First, the regime is a single forward instance; observations from different instances or contexts call for the classical and contextual formulations \citep{aswani2018inverse, besbes2023contextual}. Second, for nonlinear basis functions the model remains nonconvex and relies on a global mixed-integer solver; decomposition or cutting-plane methods \citep{bodur2022inverse, gupta2022decomposition} are the next step. Third, the dietary application is retrospective, with at most two days per person, so it does not show that the recommendations recover a habitual diet or that people would follow them. Fourth, the squared-loss guarantees assume zero-mean noise, which places observations outside $\Omega$ with positive probability at a unique linear optimum; decisions that are always feasible need a feasible-noise model, such as the duality-gap losses of \citet{chan2019inverse} and \citet{babier2021ensemble}. Extensions include food substitutions \citep{ahmadi2026similarity}, coverage guarantees for $\Theta^*(\hat z)$ through conformal inverse optimization \citep{lin2024conformal}, and integer forward problems via generator sets of the inverse-feasible cone \citep{bodur2022inverse}.

\section*{Code and Data Availability}
\label{sec:code-data}

The code and the derived diet data are provided as supplemental material in the layout of the IJOC software repository and will be released there upon acceptance; the NHANES files are public. The theorems, propositions, and corollaries of Sections~\ref{sec:setting} and \ref{sec:PIO}, apart from model-size counts, are verified in Lean~4 with Mathlib (formalization included).

\ACKNOWLEDGMENT{The authors used Claude (Anthropic) to assist in drafting and editing the text, writing and checking the experiment code, and formalizing the proofs in Lean. They verified all content and are responsible for it.}

\OneAndAHalfSpacedXII
\bibliographystyle{informs2014}
\bibliography{IO_Multipoint}
\clearpage
\ECSwitch
\renewcommand{\theHsection}{EC.\arabic{section}}%
\ECHead{Electronic Companion}

\section{Proofs}
\label{ec:proofs}

Throughout, $\|\cdot\|$ without a subscript is the Euclidean norm. We use the following standard facts about a nonempty polyhedron $\Omega = \{x : Ax \ge b\}$ \citep[Chapters 7--8]{schrijver1986theory}: (F1) if $\inf\{\theta^\top x : x \in \Omega\}$ is finite, the infimum is attained \citep[Corollary 7.1g]{schrijver1986theory}; (F2) $\Omega$ is pointed if and only if $\operatorname{rank}(A) = n$, and a nonempty pointed polyhedron $P$ has finitely many vertices $V(P) \ne \emptyset$, a pointed recession cone $\rec(P)$ with finitely many extreme rays, and $P = \conv(V(P)) + \rec(P)$, where $\rec(P)$ is the conic hull of its extreme-ray directions; (F3) $z \in \Omega$ is a vertex if and only if $\operatorname{rank}\{a_i : i \in I(z)\} = n$; (F4) for $z \in \Omega$ and $d \in \mathbb{R}^n$, $z + td \in \Omega$ for all sufficiently small $t > 0$ if and only if $a_i^\top d \ge 0$ for all $i \in I(z)$; (F5) a finitely generated cone is closed and polyhedral \citep[Corollary 7.1a]{schrijver1986theory}. We also use Farkas' lemma in the form: for vectors $a_i$, $i \in I$, and $\theta$, either $\theta \in \cone\{a_i : i \in I\}$ or there is $d$ with $a_i^\top d \ge 0$ for all $i \in I$ and $\theta^\top d < 0$, and not both.

\begin{lemma}[Optimal faces]\label{lem:optimal-face}
Let $\Omega = \{x : Ax \ge b\}$ be nonempty and pointed, and let $\theta \in \mathbb{R}^n$ be such that $v^* := \inf\{\theta^\top x : x \in \Omega\}$ is finite. Let $F^* := \{x \in \Omega : \theta^\top x = v^*\}$. Then (i) $F^*$ is nonempty and pointed, so $F^* = \conv(V(F^*)) + \rec(F^*)$ with $V(F^*) \ne \emptyset$ finite; (ii) every vertex of $F^*$ is a vertex of $\Omega$; (iii) $\theta^\top d \ge 0$ for every $d \in \rec(\Omega)$, and $\rec(F^*) = \{d \in \rec(\Omega) : \theta^\top d = 0\}$; (iv) every extreme-ray direction of $\rec(F^*)$ is an extreme-ray direction of $\rec(\Omega)$.
\end{lemma}

\begin{proof}{Proof.}
(i) $F^*$ is nonempty by (F1) and is a polyhedron, $F^* = \{x : Ax \ge b,\ \theta^\top x \ge v^*,\ -\theta^\top x \ge -v^*\}$; it contains no line because $\Omega$ contains none, so (F2) applies.
(ii) Let $x \in F^*$ be a vertex of $F^*$ and suppose $x = \tfrac12(y + w)$ with $y, w \in \Omega$. Then $\theta^\top y \ge v^*$, $\theta^\top w \ge v^*$, and $\tfrac12(\theta^\top y + \theta^\top w) = v^*$, so $\theta^\top y = \theta^\top w = v^*$ and $y, w \in F^*$; since $x$ is a vertex of $F^*$, $y = w = x$. Hence $x$ is a vertex of $\Omega$.
(iii) If $d \in \rec(\Omega)$ with $\theta^\top d < 0$, then for any $x \in \Omega$, $\theta^\top(x + td) \to -\infty$ as $t \to \infty$, contradicting finiteness of $v^*$. Next, $d \in \rec(F^*)$ if and only if $x + td \in F^*$ for all $t \ge 0$ and some (equivalently every) $x \in F^*$, i.e., $Ad \ge 0$ and $\theta^\top d = 0$.
(iv) Let $r$ be an extreme-ray direction of $\rec(F^*)$ and suppose $r = \tfrac12(s + u)$ with $s, u \in \rec(\Omega)$. By (iii), $\theta^\top s \ge 0$, $\theta^\top u \ge 0$, and $\theta^\top s + \theta^\top u = 2\theta^\top r = 0$, so $\theta^\top s = \theta^\top u = 0$ and $s, u \in \rec(F^*)$. Since $r$ is extreme in $\rec(F^*)$, $s$ and $u$ are nonnegative multiples of $r$. Hence $r$ is extreme in $\rec(\Omega)$. \Halmos
\end{proof}

\begin{proof}{Proof of Theorem~\ref{thm:generic-uniqueness}.}
Let $\theta \notin \Nex$ with $v^* = \inf\{\theta^\top x : x \in \Omega\}$ finite, and let $F^*$ be the optimal face. By Lemma~\ref{lem:optimal-face}(i), $F^* = \conv(V^*) + \cone(R^*)$ with $V^* := V(F^*) \ne \emptyset$ finite and $R^*$ the set of extreme-ray directions of $\rec(F^*)$. By Lemma~\ref{lem:optimal-face}(ii) and (iv), $V^* \subseteq V$ and every $r \in R^*$ is a positive multiple of some element of $R$. If $|V^*| \ge 2$, pick $v \ne w$ in $V^*$; both are optimal, so $\theta^\top(v - w) = 0$ with $v - w \in D$, i.e., $\theta \in \Nex$, a contradiction. If $R^* \ne \emptyset$, pick $r \in R^*$; by Lemma~\ref{lem:optimal-face}(iii), $\theta^\top r = 0$, and $r$ is a positive multiple of an element of $D$, so again $\theta \in \Nex$. Hence $|V^*| = 1$ and $R^* = \emptyset$, so $\rec(F^*) = \{0\}$ (a pointed cone with no extreme rays is $\{0\}$ by (F2)) and $F^* = \{v\}$ for the unique $v \in V^* \subseteq V$: the optimal solution is unique and is a vertex of $\Omega$.

For the probabilistic statement, $\Nex$ is a finite union of sets $\{\theta : \theta^\top d = 0\}$ with $d \ne 0$, each of probability zero by assumption, so $\mathbb{P}(\theta \in \Nex) = 0$. The assumption holds if $\theta$ has a density, because these sets are hyperplanes and have Lebesgue measure zero, and if $\theta = g/\|g\|$ for a vector $g$ with a density (so $g \ne 0$ almost surely), because $\theta^\top d = 0$ if and only if $g^\top d = 0$. If $\Omega$ is bounded, $v^*$ is finite for every $\theta$ (a continuous function on a compact set), so the forward problem has a unique optimal vertex whenever $\theta \notin \Nex$, i.e., with probability one. \Halmos
\end{proof}

\begin{proof}{Proof of Proposition~\ref{prop:uniqueness-interior}.}
(i) If $\theta = \sum_{i \in I(z)} \lambda_i a_i$ with $\lambda \ge 0$, then for every $x \in \Omega$, $\theta^\top (x - z) = \sum_{i \in I(z)} \lambda_i (a_i^\top x - b_i) \ge 0$, since $a_i^\top z = b_i$ for $i \in I(z)$ and $a_i^\top x \ge b_i$; so $z$ is optimal. Conversely, if $\theta \notin \Kc(z)$, Farkas' lemma gives $d$ with $a_i^\top d \ge 0$ for all $i \in I(z)$ and $\theta^\top d < 0$; by (F4), $z + td \in \Omega$ for small $t > 0$, and $\theta^\top (z + td) < \theta^\top z$, so $z$ is not optimal.

(ii) Suppose $\theta \in \interior \Kc(z)$ and let $x \in \Omega$ be optimal with $x \ne z$; put $u := x - z \ne 0$. For $\varepsilon > 0$ small enough, $\theta - \varepsilon u \in \Kc(z)$, so by (i) $z$ is optimal for the cost vector $\theta - \varepsilon u$ and $(\theta - \varepsilon u)^\top (x - z) \ge 0$. But $\theta^\top (x - z) = 0$ by optimality of both points, so $-\varepsilon \|u\|^2 \ge 0$, a contradiction. Hence $z$ is the unique optimal solution.

Conversely, suppose $z$ is the unique optimal solution but $\theta \notin \interior \Kc(z)$. By (i), $\theta \in \Kc(z)$. We produce $w \ne 0$ with $w^\top\theta = 0$ and $w^\top c \ge 0$ for all $c \in \Kc(z)$. If $\interior \Kc(z) = \emptyset$, the span of $\Kc(z)$ is a proper subspace, and any nonzero $w$ orthogonal to it will do. Otherwise $\theta$ is a boundary point of the closed convex set $\Kc(z)$, and the supporting hyperplane theorem \citep[Theorem 11.6]{rockafellar1970convex} gives $w \ne 0$ with $w^\top \theta \le w^\top c$ for all $c \in \Kc(z)$; taking $c = 0$ gives $w^\top \theta \le 0$, and taking $c = t\theta$ with $t \to \infty$ gives $w^\top \theta \ge 0$ (otherwise $w^\top c \to -\infty$), so $w^\top \theta = 0$ and $w^\top c \ge 0$ on $\Kc(z)$. In either case $w^\top a_i \ge 0$ for all $i \in I(z)$. By (F4), $z + tw \in \Omega$ for small $t > 0$, and $\theta^\top(z + tw) = \theta^\top z$, so $z + tw$ is a second optimal solution, a contradiction.

For the last statement, a convex cone in $\mathbb{R}^n$ has nonempty interior if and only if its linear span is $\mathbb{R}^n$ \citep[Theorem 6.2]{rockafellar1970convex}; the span of $\Kc(z)$ is the span of $\{a_i : i \in I(z)\}$, and this is $\mathbb{R}^n$ if and only if $z$ is a vertex by (F3). \Halmos
\end{proof}

\begin{proof}{Proof of Corollary~\ref{cor:non-uniqueness}.}
By Proposition~\ref{prop:uniqueness-interior}(i), the set of cost vectors for which $z$ is optimal is $\Kc(z)$, a polyhedral cone; by (F3) its generators span $\mathbb{R}^n$, so it has dimension $n$. The set depends only on $z$, hence is the same for any number of observations equal to $z$. By Proposition~\ref{prop:uniqueness-interior}(i), $(\theta, \{\lambda^k\}, \{z^k\})$ with all $z^k = z$ is feasible for $\ICO$ if and only if $\theta \in \Kc(z) \cap \Theta$.

Sign-free normalization ($n \ge 2$): $\Kc(z)$ has nonempty interior, so it contains an open ball $B(\theta_1, \varepsilon)$ with $\theta_1 \ne 0$ and $\varepsilon < \|\theta_1\|$, which avoids the origin; the radial map $\theta \mapsto \theta/\|\theta\|$ sends this ball onto a relatively open subset of the unit sphere contained in $\Kc(z) \cap \{\|\theta\| = 1\}$, which is therefore infinite because the sphere has positive dimension when $n \ge 2$.

Simplex normalization: the set $C := \Kc(z) \cap \mathbb{R}^n_+$ is a polyhedral cone contained in the nonnegative orthant, hence pointed, and $\{\theta \ge 0 : \mathbf{1}^\top\theta = 1\}$ meets every nonzero ray of $C$ in exactly one point. Thus $C \cap \{\mathbf{1}^\top\theta = 1\}$ is a compact base of $C$, of dimension $\dim C - 1$ when $C \ne \{0\}$, and it contains more than one point if and only if $\dim C \ge 2$. For $\Omega = \{x \in \mathbb{R}^2 : x_1 \ge 0,\ x_1 - x_2 \ge 0\}$ and $z = 0$, $\Kc(z) = \cone\{(1,0),(1,-1)\}$ has dimension $2$ while $C = \cone\{(1,0)\}$ has dimension $1$, so $\dim C < \dim \Kc(z)$ is possible. \Halmos
\end{proof}

Proposition~\ref{prop:param-uniqueness} counts compatible parameters before the normalization $\Theta$ is imposed; as Corollary~\ref{cor:non-uniqueness} shows in the linear case, intersecting with $\Theta$ can leave a single point even when the compatible parameters form more than one ray. The proposition's hypothesis of a one-dimensional normal cone fails at every vertex when $n \ge 2$, and its rank condition fails whenever $p > n$.

\begin{proposition}[Parameter uniqueness at a common optimum]\label{prop:param-uniqueness}
Let $z^* \in \Omega$ have a one-dimensional normal cone, $N_\Omega(z^*) = \cone\{n^*\}$ with $\|n^*\|_2 = 1$, and let $\theta \in \mathbb{R}^p$ be compatible with $z^*$ with a strictly positive multiplier, $\Phi(z^*)\theta = -\alpha n^*$, $\alpha > 0$. Then every $\theta' \ne 0$ compatible with $z^*$ (i.e., $\Phi(z^*)\theta' \in -N_\Omega(z^*)$) is a positive multiple of $\theta$ if and only if $\Phi(z^*)$ has full column rank $p$. More generally, for imputed optima $\{z^k\}_{k=1}^K$ with one-dimensional normal cones $\cone\{n^k\}$ and $\Phi(z^k)\theta = -\alpha_k n^k$, $\alpha_k > 0$, let $P_k := I_n - n^k(n^k)^\top$ and $S := \sum_k \Phi(z^k)^\top P_k \Phi(z^k)$; then $S\theta = 0$, and the compatible parameter is unique up to positive scaling if and only if $\operatorname{rank}(S) = p-1$.
\end{proposition}

\begin{proof}{Proof of Proposition~\ref{prop:param-uniqueness}.}
We prove the general statement; the single-point statement is the case $K = 1$, $z^1 = z^*$, together with the rank computation at the end. Let $\theta'$ be compatible with $\{z^k\}$: $\Phi(z^k)\theta' \in -\cone\{n^k\}$, i.e., $\Phi(z^k)\theta' = -\alpha'_k n^k$ with $\alpha'_k \ge 0$. Since $P_k n^k = n^k - n^k (n^k)^\top n^k = 0$, we get $P_k \Phi(z^k)\theta' = 0$ for every $k$, hence $\theta'^\top S \theta' = \sum_k \|P_k \Phi(z^k)\theta'\|^2 = 0$, and because $S \succeq 0$ this forces $S\theta' = 0$. Applied to $\theta$ itself this gives $S\theta = 0$, so $\operatorname{rank}(S) \le p - 1$.

If $\operatorname{rank}(S) = p-1$, then $\ker S = \operatorname{span}\{\theta\}$, so every compatible $\theta' = c\theta$ for some $c \in \mathbb{R}$; $c \ne 0$ because $\theta' \ne 0$, and $c < 0$ would give $\Phi(z^k)\theta' = |c|\alpha_k n^k$ with $\alpha_k > 0$, which lies in $-\cone\{n^k\}$ only if $n^k = 0$, impossible. Hence $c > 0$.

If $\operatorname{rank}(S) \le p - 2$, pick $\Delta \in \ker S$ not parallel to $\theta$. From $\Delta^\top S \Delta = 0$ we get $P_k \Phi(z^k)\Delta = 0$, i.e., $\Phi(z^k)\Delta = \beta_k n^k$ for some $\beta_k \in \mathbb{R}$. For $t > 0$, $\Phi(z^k)(\theta + t\Delta) = -(\alpha_k - t\beta_k) n^k$, and $0 < t < \min_k \alpha_k / \max(1, \max_k |\beta_k|)$ keeps every coefficient $\alpha_k - t\beta_k$ positive. Thus $\theta + t\Delta$ is compatible and is not a positive multiple of $\theta$.

Rank computation at a common optimum: with $z^k = z^*$ for all $k$, $S = K \Phi^\top P \Phi$ with $\Phi := \Phi(z^*)$ and $P := I_n - n^*(n^*)^\top$, and $\ker S = \ker(P\Phi) = \{\Delta : \Phi\Delta \in \operatorname{span}\{n^*\}\}$. The dimension of this preimage is $\dim \ker \Phi + \dim(\operatorname{range} \Phi \cap \operatorname{span}\{n^*\})$, and the second term equals $1$ because $\Phi\theta = -\alpha n^*$ with $\alpha > 0$ puts $n^*$ in $\operatorname{range} \Phi$. Hence $\dim \ker S = p - \operatorname{rank} \Phi + 1$, and $\operatorname{rank} S = p - 1$ if and only if $\operatorname{rank} \Phi = p$. \Halmos
\end{proof}

\begin{proof}{Proof of Theorem~\ref{thm:complexity}.}
Direct counting. $\PIO$ has variables $(z,\theta,\lambda) \in \mathbb{R}^n \times \mathbb{R}^p \times \mathbb{R}^m$; its constraints are $m$ primal inequalities \eqref{eq:PIO-primal}, $n$ stationarity equations \eqref{eq:PIO-stat}, and $m$ complementarity equations \eqref{eq:PIO-comp}. $\ICO$ has $K$ copies of $(z^k,\lambda^k)$ and one $\theta$, and $K$ copies of each constraint block. In the linear case each complementarity condition $\lambda_i(a_i^\top z - b_i) = 0$ is modeled by one binary variable (with big-$M$ bounds on $\lambda_i$ and on the slack, or as an indicator constraint), giving $m$ binaries in $\PIO$ and $Km$ in $\ILO$.

Shared multiplier for $\ILO$. Replace $\lambda^1,\ldots,\lambda^K$ by one $\lambda \ge 0$ with $\theta = A^\top\lambda$ and $\lambda_i(a_i^\top z^k - b_i) = 0$ for all $i, k$, with one binary per $i$. Every feasible point of this model is feasible for \eqref{eq:ILO} with $\lambda^k = \lambda$. Conversely, let $(\theta, \{\lambda^k\}, \{z^k\})$ be feasible for \eqref{eq:ILO} and set $\lambda := \lambda^1$. Each $z^k$ is optimal for $\min\{\theta^\top x : x \in \Omega\}$ by Proposition~\ref{prop:uniqueness-interior}(i), so $\theta^\top z^k = \theta^\top z^1 = \lambda^\top b$, the last equality by complementarity at $z^1$. Hence $\sum_i \lambda_i(a_i^\top z^k - b_i) = \theta^\top z^k - \lambda^\top b = 0$, a sum of nonnegative terms, so every term vanishes and $\lambda$ is complementary to every $z^k$. The two models therefore have the same feasible $(\theta, \{z^k\})$ and the same optimal value. \Halmos
\end{proof}

\begin{proof}{Proof of Proposition~\ref{prop:aggregation}.}
Squared loss: $\sum_k \|x^k - z\|^2 = \sum_k \|x^k - \bar x\|^2 + 2\sum_k (x^k - \bar x)^\top(\bar x - z) + K\|\bar x - z\|^2$, and the middle term vanishes because $\sum_k (x^k - \bar x) = 0$. Absolute loss: $\|x^k - z\|_1 = \sum_i |x^k_i - z_i|$, so the sum over $k$ separates across $i$; for fixed $i$, $h_i(t) := \sum_k |x^k_i - t|$ is a sum of convex piecewise-linear functions with breakpoints at the values $x^k_i$, its slope on an interval between consecutive sorted values is (number of $x^k_i$ below) minus (number above), and it is minimized at any median of $\{x^k_i\}_k$; the epigraph of each $h_i$ is described by $K+1$ affine inequalities in $(z_i, t_i)$, which is the linear-programming model with $n$ auxiliary variables. \Halmos
\end{proof}

\begin{proof}{Proof of Theorem~\ref{thm:param-set}.}
We first verify the representation of the normal cone. If $v = \sum_{i \in I(z)} \mu_i \nabla g_i(z)$ with $\mu \ge 0$, then for $x \in \Omega$ convexity gives $\nabla g_i(z)^\top (x - z) \le g_i(x) - g_i(z) = g_i(x) \le 0$ for $i \in I(z)$, so $v^\top(x - z) \le 0$ and $v \in N_\Omega(z)$. Conversely, if $v \in N_\Omega(z)$ then $z$ minimizes the linear function $-v^\top x$ over $\Omega$, and the KKT theorem for convex programs under Slater's condition \citep[Theorem 28.2 and Corollary 28.3.1]{rockafellar1970convex} gives $\mu \in \mathbb{R}^m_+$ with $-v + \sum_i \mu_i \nabla g_i(z) = 0$ and $\mu_i g_i(z) = 0$, i.e., $v \in \cone\{\nabla g_i(z) : i \in I(z)\}$. Now fix $z \in \Omega$ and $\theta \in \Theta$. Since $f(\cdot,\theta)$ and the $g_i$ are convex and differentiable and Slater's condition holds, $z$ is optimal for $\FO(\theta,\Omega)$ if and only if there is $\lambda \in \mathbb{R}^m_+$ with $\Phi(z)\theta + \sum_i \lambda_i \nabla g_i(z) = 0$ and $\lambda_i g_i(z) = 0$ for all $i$ (same theorem). Complementarity restricts $\lambda$ to be supported on $I(z)$, so the stationarity equation says exactly $\Phi(z)\theta \in -\cone\{\nabla g_i(z) : i \in I(z)\} = -N_\Omega(z)$; conversely any representation of $-\Phi(z)\theta$ as a nonnegative combination of the active gradients provides a valid $\lambda$. This proves \eqref{eq:theta-set}. The set is the intersection of the compact set $\Theta$ with the preimage of the closed convex cone $-N_\Omega(z)$ under the linear map $\theta \mapsto \Phi(z)\theta$, hence compact, and convex whenever $\Theta$ is. The cone $-N_\Omega(z)$ is finitely generated, hence polyhedral by (F5), say $\{y : Cy \le 0\}$; its preimage $\{\theta : C\Phi(z)\theta \le 0\}$ is a polyhedron, and intersecting with a polyhedral $\Theta$ gives a polyhedron. In the linear case $\Phi(z) = I_n$ and $\nabla g_i = -a_i$, so the condition reads $\theta \in \cone\{a_i : i \in I(z)\} = \Kc(z)$. \Halmos
\end{proof}

\begin{proof}{Proof of Proposition~\ref{prop:linear-structure}.}
(i) In the linear case the constraints \eqref{eq:PIO-primal}--\eqref{eq:PIO-dual} read $z \in \Omega$, $\theta = A^\top\lambda$, $\lambda \ge 0$, $\lambda_i(a_i^\top z - b_i) = 0$, $\|\theta\| = 1$; so $(z,\theta,\lambda)$ is feasible if and only if $z \in \Omega$, $\|\theta\| = 1$, and $\theta$ is a nonnegative combination of the active normals, i.e., $\theta \in \Kc(z)$, which by Proposition~\ref{prop:uniqueness-interior}(i) says exactly that $z$ is optimal for $\theta$. Hence the feasible $z$ are exactly the points of $\Zstar$ (no constraint qualification is needed here). Let $z \in \Zstar$, optimal for some $\theta$ with $\|\theta\| = 1$; then $\theta \in \Kc(z)$ and, since $\theta \ne 0$, $I(z) \ne \emptyset$ and $z \in F_i$ for every $i \in I(z)$. Conversely, if $z \in F_i$ then $i \in I(z)$, $a_i/\|a_i\| \in \Kc(z) \cap \Theta$, and $z$ is optimal for it, so $z \in \Zstar$. Hence $\Zstar = \bigcup_i F_i$, the minimum of the loss over the feasible set equals the minimum over $i$ of the minimum over $F_i$ (empty $F_i$ contributing $+\infty$), and for a minimizer $z^*$ on $F_{i^*}$ the triple $(z^*, a_{i^*}/\|a_{i^*}\|, e_{i^*}/\|a_{i^*}\|)$ is feasible. For the squared and absolute losses the objective is continuous and coercive and each nonempty $F_i$ is closed, so every inner minimum is attained. Each inner problem minimizes a convex function over the polyhedron $F_i$; for the squared loss it is the convex quadratic program $\min\{K\|z - \bar x\|^2 : z \in F_i\}$ with $n$ variables and the rows of $F_i$ as constraints, and for the absolute loss it is a linear program in $(z, t) \in \mathbb{R}^n \times \mathbb{R}^n$ with the $n(K+1)$ epigraph inequalities of Proposition~\ref{prop:aggregation} and the rows of $F_i$.

(ii) Let $(\theta, \{\lambda^k\}, \{z^k\})$ be feasible for \eqref{eq:ILO} with $\|\theta\| = 1$. Every $z^k$ is optimal for $\theta$, so the average $\bar z := \frac1K \sum_k z^k$ lies in $\Omega$ and is optimal for $\theta$ as well. By Proposition~\ref{prop:uniqueness-interior}(i), $\theta \in \Kc(\bar z)$, and $\theta \ne 0$ gives $I(\bar z) \ne \emptyset$. For $i \in I(\bar z)$, $0 = a_i^\top \bar z - b_i = \frac1K \sum_k (a_i^\top z^k - b_i)$ is an average of nonnegative terms, so $a_i^\top z^k = b_i$ for every $k$, and all $z^k$ lie in $F_i$. Thus the $\ILO$ value is at least $\min_i \sum_k \min_{z^k \in F_i} \ell(x^k - z^k)$. Conversely, for any $i$ and any $z^1,\ldots,z^K \in F_i$, the choice $\theta = a_i/\|a_i\|$, $\lambda^k = e_i/\|a_i\|$ is feasible for \eqref{eq:ILO}, which gives the reverse inequality. The inner problem is a sum over $k$ of functions of $z^k$ alone, so it separates.

(iii) Let $q := \proj_\Omega(\bar x)$, the unique nearest point of the closed convex set $\Omega$ to $\bar x$ \citep[Proposition 1.1.9]{bertsekas2009convex}. Since $\bar x \notin \Omega$, $q \ne \bar x$, and the projection is characterized by $(\bar x - q)^\top (x - q) \le 0$ for all $x \in \Omega$; hence $q$ minimizes $(q - \bar x)^\top x$ over $\Omega$, i.e., $q$ is optimal for the nonzero cost vector $\theta = (q - \bar x)/\|q - \bar x\| \in \Theta$, and $q \in \Zstar$. By part (i) and Proposition~\ref{prop:aggregation}, $\PIO$ minimizes $\|z - \bar x\|$ over $\Zstar \subseteq \Omega$, and $\|z - \bar x\| \ge \|q - \bar x\|$ for every $z \in \Omega$ with equality only at $z = q$. Hence $\hat z = q$, uniquely.

(iv) If $i \in I(z_R)$ then $z_R \in F_i \subseteq \Zstar$ by part (i), and $z_R$ minimizes the loss over $\Omega \supseteq \Zstar$, so it minimizes it over $\Zstar$. If no constraint is active at $z_R$, then $a_i^\top z_R > b_i$ for all $i$, so a neighborhood of $z_R$ lies in $\Omega$ and $z_R$ is a local, hence (by convexity) global, minimizer of the loss over $\mathbb{R}^n$; it is then an unconstrained minimizer lying in $\Omega$. \Halmos
\end{proof}

\begin{lemma}[Closedness of the rationalizable set]\label{lem:Zstar-closed}
Under Assumption~\ref{assump:standing}, $\Zstar$ is closed, and for either loss the problem \eqref{eq:PIO} has a minimizer whenever $\Zstar \ne \emptyset$.
\end{lemma}

\begin{proof}{Proof.}
Let $z_\nu \in \Zstar$ with $z_\nu \to z$. By Theorem~\ref{thm:param-set} there are $\theta_\nu \in \Theta$ and $\lambda_\nu \in \mathbb{R}^m_+$ with $\Phi(z_\nu)\theta_\nu + \sum_i \lambda_{\nu,i} \nabla g_i(z_\nu) = 0$ and $\lambda_{\nu,i} g_i(z_\nu) = 0$. Since $\Omega$ is closed, $z \in \Omega$. We claim $\{\lambda_\nu\}$ is bounded. If not, pass to a subsequence with $\|\lambda_\nu\|_1 \to \infty$ and $\mu_\nu := \lambda_\nu/\|\lambda_\nu\|_1 \to \mu$, where $\mu \ge 0$ and $\|\mu\|_1 = 1$. For $i \notin I(z)$, $g_i(z) < 0$, so $g_i(z_\nu) < 0$ for large $\nu$ by continuity and $\lambda_{\nu,i} = 0$; hence $\mu_i = 0$ for $i \notin I(z)$. Dividing the stationarity equation by $\|\lambda_\nu\|_1$ and letting $\nu \to \infty$ (using boundedness of $\Theta$ and continuity of $\Phi(\cdot)$ and $\nabla g_i$) gives $\sum_{i \in I(z)} \mu_i \nabla g_i(z) = 0$. Let $\hat x$ be a Slater point. By convexity of $g_i$, $g_i(\hat x) \ge g_i(z) + \nabla g_i(z)^\top (\hat x - z) = \nabla g_i(z)^\top(\hat x - z)$ for $i \in I(z)$, so $\sum_{i \in I(z)} \mu_i g_i(\hat x) \ge \bigl(\sum_{i \in I(z)} \mu_i \nabla g_i(z)\bigr)^\top (\hat x - z) = 0$. But the left side is strictly negative because $\mu \ge 0$, $\mu \ne 0$, and $g_i(\hat x) < 0$ for all $i$, a contradiction. So $\{\lambda_\nu\}$ is bounded, and along a subsequence $\theta_\nu \to \theta \in \Theta$ (compactness), $\lambda_\nu \to \lambda \ge 0$; passing to the limit in the stationarity and complementarity equations shows $(z,\theta,\lambda)$ satisfies the KKT conditions, so $z \in \Zstar$.

Existence: the objective $z \mapsto \sum_k \ell(x^k - z)$ is continuous and tends to $+\infty$ as $\|z\| \to \infty$ for both losses, so its sublevel set through any point of $\Zstar$ intersected with the closed set $\Zstar$ is compact and nonempty, and a minimizer exists. \Halmos
\end{proof}

\begin{lemma}[Medians]\label{lem:median}
Let $\eta$ be a real random variable with $\mathbb{E}|\eta| < \infty$ whose unique median is $0$, i.e., $\mathbb{P}(\eta \le 0) \ge \tfrac12$, $\mathbb{P}(\eta \ge 0) \ge \tfrac12$, and no other $t$ has both properties. Then $t \mapsto \mathbb{E}|\eta - t|$ is finite, convex, and has $0$ as its unique minimizer.
\end{lemma}

\begin{proof}{Proof.}
Finiteness and convexity are immediate. Fix $t > 0$. For every real $s$, $|s - t| - |s| = t$ if $s \le 0$, $= t - 2s$ if $0 < s < t$, and $= -t$ if $s \ge t$. Taking expectations,
\[
\mathbb{E}|\eta - t| - \mathbb{E}|\eta| = t\,\mathbb{P}(\eta \le 0) + \mathbb{E}\bigl[(t - 2\eta)\mathbf{1}\{0 < \eta < t\}\bigr] - t\,\mathbb{P}(\eta \ge t).
\]
On $\{0 < \eta < t\}$ we have $t - 2\eta > -t$, so the middle term is at least $-t\,\mathbb{P}(0 < \eta < t)$, and the right side is at least $t\,[\mathbb{P}(\eta \le 0) - \mathbb{P}(\eta > 0)] \ge 0$ because $\mathbb{P}(\eta \le 0) \ge \tfrac12 \ge \mathbb{P}(\eta > 0)$. Equality throughout would require $\mathbb{P}(0 < \eta < t) = 0$ and $\mathbb{P}(\eta \le 0) = \tfrac12$; but then for any $t' \in (0,t)$, $\mathbb{P}(\eta \le t') = \mathbb{P}(\eta \le 0) = \tfrac12$ and $\mathbb{P}(\eta \ge t') = \mathbb{P}(\eta > 0) = \tfrac12$, so $t'$ would be another median, contradicting uniqueness. Hence $\mathbb{E}|\eta - t| > \mathbb{E}|\eta|$ for $t > 0$; the case $t < 0$ is symmetric (apply the argument to $-\eta$, whose unique median is also $0$). \Halmos
\end{proof}

\begin{proof}{Proof of Theorem~\ref{thm:consistency}.}
By Theorem~\ref{thm:param-set} and Lemma~\ref{lem:Zstar-closed}, $\hat z_K$ is a minimizer of $z \mapsto \sum_k \ell(x^k - z)$ over the closed set $\Zstar$, which contains $z_0$ by Assumption~\ref{assump:DGP}. The bounds below hold for every realization of the data and every choice of minimizer, so they hold for any measurable selection $\hat z_K$, which is all the probabilistic statements require.

(i) By Proposition~\ref{prop:aggregation}, $\hat z_K$ minimizes $\|z - \bar x_K\|$ over $\Zstar$, so $\|\hat z_K - \bar x_K\| \le \|z_0 - \bar x_K\|$ and, by the triangle inequality, $\|\hat z_K - z_0\| \le \|\hat z_K - \bar x_K\| + \|\bar x_K - z_0\| \le 2\|\bar x_K - z_0\|$, for every realization. Next, $\bar x_K - z_0 = \frac1K \sum_k \xi^k$, and since the $\xi^k$ are independent with mean zero, $\mathbb{E}\|\sum_k \xi^k\|^2 = \sum_k \mathbb{E}\|\xi^k\|^2 = K\sigma^2$; hence $\mathbb{E}\|\bar x_K - z_0\| \le (\mathbb{E}\|\bar x_K - z_0\|^2)^{1/2} = \sigma/\sqrt K$ by Jensen's inequality, and $\mathbb{E}\|\hat z_K - z_0\| \le 2\sigma/\sqrt K$. Finally, $\bar x_K \to z_0$ almost surely by the strong law of large numbers ($\mathbb{E}\|\xi^k\| \le \sigma < \infty$), so $\hat z_K \to z_0$ almost surely.

(ii) Let $L_K(z) := \frac1K \sum_{k=1}^K \|x^k - z\|_1$ and $L(z) := \mathbb{E}\|x^1 - z\|_1 = \sum_{i=1}^n \mathbb{E}|\xi^1_i - (z_i - z_{0,i})|$. By Lemma~\ref{lem:median} applied coordinatewise, $L$ is finite and convex on $\mathbb{R}^n$ and $z_0$ is its unique minimizer; in particular $z_0$ is the unique minimizer of $L$ over $\Zstar$. For each fixed $z$, $\|x^k - z\|_1$, $k = 1, 2, \ldots$, are i.i.d.\ with finite mean, so $L_K(z) \to L(z)$ almost surely by the strong law of large numbers. Let $Q$ be a countable dense subset of $\mathbb{R}^n$; the countable union of null events shows that, almost surely, $L_K(q) \to L(q)$ for every $q \in Q$. Each $L_K$, and $L$, is $1$-Lipschitz for $\|\cdot\|_1$, because $\bigl|\|x - z\|_1 - \|x - z'\|_1\bigr| \le \|z - z'\|_1$. As $\Omega$ is bounded and $\Zstar \subseteq \Omega$ is closed (Lemma~\ref{lem:Zstar-closed}), $\Zstar$ is compact. Given $\varepsilon > 0$, cover $\Zstar$ by finitely many $\|\cdot\|_1$-balls of radius $\varepsilon$ centered at points of $Q$. On the almost-sure event above, $|L_K(q) - L(q)| < \varepsilon$ at these finitely many centers for all large $K$, and the Lipschitz bound then gives $|L_K(z) - L(z)| < 3\varepsilon$ for every $z \in \Zstar$. Hence $\sup_{z \in \Zstar} |L_K(z) - L(z)| \to 0$ almost surely.

Fix $\varepsilon > 0$ and let $\delta := \min\{L(z) - L(z_0) : z \in \Zstar,\ \|z - z_0\| \ge \varepsilon\}$; the set is compact and $L$ is continuous, so the minimum is attained and $\delta > 0$ by uniqueness of the minimizer (if the set is empty there is nothing to prove). On the almost-sure event above, for all large $K$ we have $\sup_{\Zstar}|L_K - L| < \delta/2$, and then $L(\hat z_K) < L_K(\hat z_K) + \delta/2 \le L_K(z_0) + \delta/2 < L(z_0) + \delta$, which forces $\|\hat z_K - z_0\| < \varepsilon$. Since $\varepsilon$ was arbitrary, $\hat z_K \to z_0$ almost surely. \Halmos
\end{proof}

\begin{proof}{Proof of Theorem~\ref{thm:cone-recovery}.}
(i) For $i \notin I(z_0)$, $a_i^\top z_0 > b_i$, so by continuity there is a neighborhood $U$ of $z_0$ on which $a_i^\top z > b_i$ for all $i \notin I(z_0)$; for $z \in \Omega \cap U$ this gives $I(z) \subseteq I(z_0)$, hence $\Kc(z) \subseteq \Kc(z_0)$ and, by Theorem~\ref{thm:param-set}, $\Theta^*(z) = \Kc(z) \cap \Theta \subseteq \Theta^*(z_0)$. On the almost-sure event $\hat z_K \to z_0$, $\hat z_K \in U$ for all large $K$, and the containment follows.

(ii) Let $z_\nu \to z$ in $\Omega$, $\theta_\nu \in \Theta^*(z_\nu)$, $\theta_\nu \to \theta$. By Theorem~\ref{thm:param-set} there are $\lambda_\nu \ge 0$ with $\Phi(z_\nu)\theta_\nu + \sum_i \lambda_{\nu,i}\nabla g_i(z_\nu) = 0$ and $\lambda_{\nu,i} g_i(z_\nu) = 0$. Exactly as in the proof of Lemma~\ref{lem:Zstar-closed}, Slater's condition forces $\{\lambda_\nu\}$ to be bounded, and passing to the limit along a convergent subsequence shows that $(z,\theta,\lambda)$ satisfies the KKT conditions, with $\theta \in \Theta$ because $\Theta$ is closed. Hence $\theta \in \Theta^*(z)$. Outer semicontinuity at $z_0$ applied along the almost-sure sequence $\hat z_K \to z_0$ gives $\limsup_K \Theta^*(\hat z_K) \subseteq \Theta^*(z_0)$, where $\limsup$ is the set of limits of convergent sequences $\theta_K \in \Theta^*(\hat z_K)$ along subsequences.

If $\Theta^*(z_0) = \{\theta_0\}$, every sequence $\hat\theta_K \in \Theta^*(\hat z_K)$ lies in the compact set $\Theta$ and all its cluster points belong to $\Theta^*(z_0) = \{\theta_0\}$, so $\hat\theta_K \to \theta_0$. \Halmos
\end{proof}

\begin{proof}{Proof of Theorem~\ref{thm:twostage}.}
In the linear case Theorem~\ref{thm:consistency} holds without Slater's condition: its proof uses Assumption~\ref{assump:standing} only to know that $\hat z_K$ minimizes the loss over the closed set $\Zstar$, which Proposition~\ref{prop:linear-structure}(i) gives directly ($\Zstar = \bigcup_i F_i$).

(i) Let $\|\hat z - z_0\|_2 < r/2$. For $i \in I(z_0)$, $s_i(\hat z) = a_i^\top(\hat z - z_0)/\|a_i\|_2 \le \|\hat z - z_0\|_2 < r/2$; for $i \notin I(z_0)$, $s_i(\hat z) \ge s_i(z_0) - \|\hat z - z_0\|_2 > r/2$. So the scan meets every row of $I(z_0)$ before any other row. The rows kept after any prefix of the scan form a maximal linearly independent subset of that prefix, and the rows of $I(z_0)$ have rank $n$ because $z_0$ is a vertex (F3); hence the scan keeps $n$ rows of $I(z_0)$ and stops before any other row. These rows are linearly independent and active at $z_0$, so $v = z_0$. Under the squared loss $\|\hat z_K - z_0\|_2 \le 2\|\bar x_K - z_0\|_2$ (proof of Theorem~\ref{thm:consistency}(i)), so $v_K \ne z_0$ implies $\|\bar x_K - z_0\|_2 \ge r/4$, which by Chebyshev's inequality has probability at most $16\,\mathbb{E}\|\bar x_K - z_0\|_2^2/r^2 = 16\sigma^2/(Kr^2)$. Under the absolute loss, $\hat z_K \to z_0$ almost surely by Theorem~\ref{thm:consistency}(ii).

(ii) If $v_K = z_0$, Proposition~\ref{prop:aggregation} gives $\mathcal L_K(z_0) - \mathcal L_K(\hat z_K) = K\|z_0 - \bar x_K\|_2^2 - K\|\hat z_K - \bar x_K\|_2^2 \le K\|\bar x_K - z_0\|_2^2$, so the guard accepts $z_0$ whenever $K\|\bar x_K - z_0\|_2^2 \le \tau_K$; this gives the bound. Here $\tau_K = \hat\sigma_K^2 \log K$ with $\hat\sigma_K^2 := \frac1K\sum_k \|x^k - \bar x_K\|_2^2 \to \sigma^2$ almost surely. If $\sigma = 0$ the second term is zero; otherwise Markov's inequality and $\mathbb{E}\,K\|\bar x_K - z_0\|_2^2 = \sigma^2$ give $\mathbb{P}(K\|\bar x_K - z_0\|_2^2 > \tau_K) \le 2/\log K + \mathbb{P}(\hat\sigma_K^2 < \sigma^2/2) \to 0$.

(iii) Let $V$ be the finite vertex set of $\Omega$, so $z_0 \notin V$. As $z_0 \in \Zstar$, $\mathcal L_K(\hat z_K) \le \mathcal L_K(z_0)$, so $\mathcal L_K(w) - \mathcal L_K(\hat z_K) \ge \mathcal L_K(w) - \mathcal L_K(z_0)$ for every $w \in V$. By the strong law of large numbers, $[\mathcal L_K(w) - \mathcal L_K(z_0)]/K$ converges almost surely to $\|w - z_0\|_2^2 > 0$ under the squared loss, where it equals $\|w - \bar x_K\|_2^2 - \|z_0 - \bar x_K\|_2^2$, and to $L(w) - L(z_0) > 0$ under the absolute loss, with $L$ as in the proof of Theorem~\ref{thm:consistency}(ii) (the summands are bounded by $\|w - z_0\|_1$). As $V$ is finite and $\tau_K/K \to 0$, almost surely $\mathcal L_K(w) - \mathcal L_K(\hat z_K) > \tau_K$ for every $w \in V$ and all large $K$. A rounded point $v \in \Omega$ has $n$ linearly independent active rows, so $v \in V$ by (F3), and the guard rejects it. For the $\tau_K$ of Section~\ref{subsec:twostage}, $\tau_K/K = (\log K/K)\min_z\mathcal L_K(z)/K \to 0$ because $\min_z \mathcal L_K(z)/K \le \mathcal L_K(z_0)/K$ converges almost surely.

Finally, $\Theta^*(z_0) = \Kc(z_0) \cap \Theta$ by Theorem~\ref{thm:param-set}. \Halmos
\end{proof}

\section{Additional Figures and Tables} \label{sec:supp-figs}

Figure~\ref{fig:ksweep-theta} shows what the recovered cost vector does and does not tell us in the scaling study of Section~\ref{subsec:Ksweep}. Table~\ref{Table:SubsetNutrients} gives the nutrient bounds and the recalled intake of the focus subgroup of Section~\ref{Section:Application}.

\begin{figure}[htbp]
    \centering
    \includegraphics[width=\textwidth]{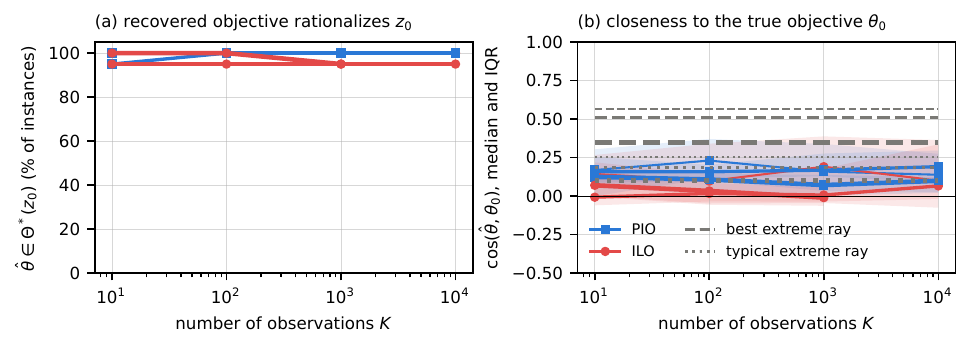}
    \caption{What the recovered cost vector tells us, linear instances of Section~\ref{subsec:Ksweep}, from the face normals returned by $\PIO$-enum and $\ILO$-enum, 20 instances per cell ($\ILO$-enum at $n = 50$, $K = 10^4$ did not finish; line width increases with $n$). (a)~Share of instances in which the recovered $\hat\theta$ lies in the true parameter cone $\Theta^*(z_0)$, i.e., rationalizes the true optimum. (b)~Cosine between $\hat\theta$ and the true $\theta_0$: median and interquartile band for $\PIO$ and $\ILO$, against two references computed from the instances alone: the best and the typical extreme ray of $\Theta^*(z_0)$. Both models return a valid ray of the cone in 95--100\% of instances per cell, and neither gets closer to $\theta_0$ as $K$ grows, as Corollary~\ref{cor:non-uniqueness} predicts at a vertex.}
    \label{fig:ksweep-theta}
\end{figure}

\begin{table}[!htbp]
\footnotesize\renewcommand{\baselinestretch}{1.0}\selectfont
\begin{center}\renewcommand{\arraystretch}{1.15}
\caption{Nutrient bounds of the feasible set $\Omega$ (DASH trial goals for a 1,600-kcal plan, \citealt{nhlbi2006dash}), recalled daily intake of the focus subgroup (NHANES day totals; 243 women, 460 recall days), share of recall days outside the bound, and correlation between the food-group model $Nx$ and the recalled totals.}
\label{Table:SubsetNutrients}
\begin{tabular}{l|rr|rr|r|r}
\hline\hline
& \multicolumn{2}{c|}{\textbf{Bounds}} & \multicolumn{2}{c|}{\textbf{Recalled intake}} & \textbf{Days} & \\
\textbf{Nutrient} & \textbf{Lower} & \textbf{Upper} & \textbf{Mean} & \textbf{Median} & \textbf{outside (\%)} & \textbf{Corr.} \\
\hline
Energy (kcal) & 1,440 & 1,760 & 1,635 & 1,528 & 80 & 0.87 \\
Carbohydrate (g) & -- & 220 & 197 & 182 & 35 & 0.84 \\
Protein (g) & 72.0 & -- & 66.3 & 61.9 & 63 & 0.92 \\
Total fat (g) & -- & 48.0 & 65.5 & 60.0 & 65 & 0.80 \\
Saturated fat (g) & -- & 10.7 & 20.5 & 18.4 & 78 & 0.76 \\
Cholesterol (mg) & -- & 150 & 263 & 210 & 65 & 0.85 \\
Sodium (mg) & -- & 2,300 & 2,934 & 2,636 & 62 & 0.85 \\
Potassium (mg) & 3,581 & -- & 2,200 & 2,049 & 90 & 0.90 \\
Calcium (mg) & 952 & -- & 732 & 642 & 75 & 0.74 \\
Magnesium (mg) & 381 & -- & 252 & 232 & 88 & 0.80 \\
Dietary fiber (g) & 22.9 & -- & 15.1 & 13.1 & 82 & 0.81 \\
Total sugars (g) & -- & -- & 79.1 & 68.6 & 0 & 0.67 \\
\hline\hline
\end{tabular}
\end{center}
\end{table}

\section{Diet Recommendation Problem: Additional Data} \label{Appendix:diet_data}
All data of Section~\ref{Section:Application} are rebuilt from the public NHANES files by \texttt{scripts/run\_nhanes.py}. The focus subgroup consists of women aged 51 to 75 who answer yes to ``ever told you had high blood pressure'' (BPQ020) and to ``ever told you have prediabetes'' (DIQ160) or ``doctor told you have diabetes'' (DIQ010); 84 of the 243 women report prediabetes and 159 diabetes. A recall day is used if NHANES flags it as complete and reliable (DR1DRSTZ or DR2DRSTZ equal to 1). Food codes beginning with 20, 67, and 76 (baby foods), 44 (carob products), 93 (alcoholic beverages), and 94 (plain water) are excluded; serving sizes are those of the open dataset of \citet{ahmadi2020opensource}.

\end{document}